\documentclass[reqno,a4paper,11pt]{amsart}
\usepackage{amsmath,amsthm,amssymb,color,graphicx,bm,mathtools}
\usepackage{mathrsfs}
\usepackage[left=2.4cm,right=2.4cm,top=3cm,bottom=3cm,bindingoffset=0cm]{geometry}
\usepackage{algorithm}
\usepackage{algorithmic}
\usepackage{microtype}
\usepackage{emptypage}
\usepackage{indentfirst}
\usepackage{CJKutf8}
\usepackage{enumitem}
\usepackage{lmodern}
\usepackage{multicol}
\usepackage[bookmarks,colorlinks=true,linkcolor=blue]{hyperref}

\newcommand{\norm}[1]{\left\lVert#1\right\rVert}
\newcommand{\scal}[2]{\langle #1,#2\rangle}

\usepackage[all, pdf]{xy}
\usepackage{subfig}
\usepackage{enumitem}
\usepackage{color}
\usepackage{graphicx,booktabs}
\newcommand{\numberset}{\mathbb}
\newcommand{\N}{\numberset{N}} 
\newcommand{\R}{\numberset{R}} 

\newcommand{\cscal}{C^{0}_{\text{\tiny sc}}}

\DeclareMathOperator{\supp}{supp}
\DeclareMathOperator*{\esssup}{ess\,sup}

\newtheorem{theorem}{Theorem}

\newtheorem{definition}[theorem]{Definition}

\newtheorem{remark}[theorem]{Remark}

\begin{document}

\allowdisplaybreaks

\author[M. Bonafini]{Mauro Bonafini}\author[V.P.C. Le]{Van Phu Cuong Le}\author[R. Molinarolo]{Riccardo Molinarolo}

\address[M.\ Bonafini]{Dipartimento di Informatica, Universit\`a degli Studi di Verona, Strada Le Grazie, 37134 Verona, Italy}
\email{bonafini.mauro@univr.it}

\address[V.P.C.\ Le]{Institut f\"ur Mathematik, Universit\"at Heidelberg, Im Neuenheimer Feld 205, 69120 Heidelberg, Germany}
\email{cuong.le@uni-heidelberg.de}

\address[R.\ Molinarolo]{Dipartimento di Informatica, Universit\`a degli Studi di Verona, Strada Le Grazie, 37134 Verona, Italy}
\email{riccardo.molinarolo@univr.it}

\title[On fractional semilinear wave equations in non-cylindrical domains]{On fractional semilinear wave equations in non-cylindrical domains}

\begin{abstract}
	In this paper, we investigate a class of semilinear wave equations in non-cylindrical time-dependent domains, subject to exterior homogeneous Dirichlet conditions. Under mild regularity and monotonicity assumptions on the evolving spatial domains, we establish existence of weak solutions by two different methods: a constructive time-discretization scheme and a penalty approach. The analysis applies to nonlocal fractional Laplacians and potentials with Lipschitz continuous gradient, and to vector-valued maps.
\end{abstract}

\maketitle


\bigskip

	
\tableofcontents

\section{Introduction}
Linear, semilinear, and non-linear wave equations have been widely investigated in the mathematical literature, with particular attention devoted to existence, well-posedness, regularity, and singular limits. Among the many significant contributions, we mention the classical monographs \cite{hormander1997lectures,lax2006hyperbolic} and the seminal works on the topics \cite{ColomboHaffter, del2020interface, Grillakis, Maruo85, Schatzman78, SeTi} and references therein.

In this paper, we consider non-local equations on non-cylindrical domains. Hyperbolic PDEs on time-dependent domains model wave propagation, fluid-structure interaction, fluid dynamics, extensible beam dynamics, cavity resonators, and many other problems in mathematical physics (see the recent survey \cite{knobloch2015problems}). In particular, the problem addressed here naturally arises from elastodynamic models in mechanics, where a thin film is initially glued to a rigid substrate and gradually peeled off. In that context, it is natural to model the debonded region with a time-dependent, non-decreasing domain and to assume that the displacement satisfies a classical wave equation (see \cite{dal2016existence, LazzaroniMolinaroloRivaSolombrino, LAZZARONI2022112822} and references therein). Hence, considering a non-local operator with possible interactions between the displacement and the external bonded region is a natural extension of this setup. Time-dependent domains also arise in dynamic fracture models, where the crack evolution induces an expanding time-dependent set (cf.\ \cite{dal2011existence, dal2017wave, dal2002model}). Furthermore, at least for cylindrical domains, non-local non-linear wave equations with a double-well potential are relevant in the study of evolving interfaces or defects, such as minimal surfaces in Minkowski space (see \cite{bellettini2010time, del2020interface, jerrard2011defects, Neu, SmailyJerrad}).

In material science, a very similar model arises in the context of peridynamics, a non-local elasticity theory first introduced by Silling in \cite{MR1727557}. The mathematical consistency, the numerical analysis and the qualitative properties have been addressed by several authors: far from been exhaustive, we mention \cite{MR4648534,MR4500878,MR3884619} as well as \cite{MR3297136,MR3403446} and references therein.

\smallskip

Our work is motivated by the following key objective. We aim to extend the existence results for fractional wave equations in \cite{bonafini2022weak, bonafini2021obstacle, bonafini2019variational} to the non-cylindrical setting, and, in doing so, to extend to the fractional setting the results in non-cylindrical domains presented in \cite{LazzaroniMolinaroloRivaSolombrino, LAZZARONI2022112822}.

The literature contains several results on hyperbolic equations in non-cylindrical domains, obtained through different methods and under various assumptions on the domain evolution. Among these contributions, in \cite{bernardi1998some} the authors present an abstract formulation together with regularization techniques for the associated operators. Notable contributions include diffeomorphism-based methods that recast the problem in cylindrical domains (see \cite{cooper1973nonlinear, da1990existence, dal2016existence, dal2017wave, dal2019cauchy, ma2017dynamics, sikorav1990linear}), with diffeomorphisms being a suitable option due to the local nature of the operators considered in those works.
Furthermore, we refer to \cite{alphonse2015abstract, bernardi2001variational, CalvoNovagaOrlandi, gianazza1996abstract, horvath2020exactly} and references therein for various approaches on different evolution equations (parabolic, Schrödinger, Navier-Stokes, etc.) in non-cylindrical domains (see also \cite{dalla2023multi, dalla2025shape} for shape analysis results for parabolic equations via potential theory).

\smallskip
In this paper, we focus on a semilinear fractional wave equation defined in a non-cylindrical domain $\mathcal{O} \subset [0,T] \times \mathbb{R}^d$, where $T > 0$ represents a finite time horizon. Time slices $\{\Omega_t\}_{t \in [0,T]}$ of $\mathcal{O}$ will describe our evolving domains. In this context, it is customary to impose some form of continuity on the domain evolution, such as Lipschitz continuity, Hölder continuity, or absolute continuity (see \cite{bonaccorsi2001variational, paronetto2013existence, savare1997parabolic}, respectively). Moreover, monotonicity assumptions (most notably expanding domains) are typical (see \cite{bernardi2001variational, gianazza1996abstract}), and will be part of our working assumptions (for details on the geometrical framework, see Section \ref{sec:problem}). We consider the general hyperbolic equation
\begin{equation}\label{eq:PDE}
    \ddot{u}+(-\Delta)^s u  + \nabla W(u) = f \qquad\text{in }\mathcal{O},
\end{equation}
where $(-\Delta)^s$ represents the fractional $s$-Laplacian, $W$ is a potential such that $\nabla W$ is Lipschitz continuous and $f$ is a given forcing term. Our aim is to prove existence of suitably defined weak solutions to \eqref{eq:PDE} coupled with homogeneous Dirichlet boundary conditions and for given initial time conditions.

Two different existence proofs are presented. The first one is based on a time-discretization scheme, inspired by the approach of Calvo, Novaga, and Orlandi \cite{CalvoNovagaOrlandi} (see also \cite{LazzaroniMolinaroloRivaSolombrino}). The second proof relies on the penalty method introduced by J.-L. Lions \cite{Lions64}, which has since been widely applied to evolution systems on time-dependent domains, including wave equations \cite{Lopez2015Remarks, zolesio1990approximation} and quasilinear hyperbolic $p$-Laplacian problems \cite{FerreiraRaposoSantos2006}.

\smallskip
The paper is organized as follows: in Section \ref{sec:pre not} we briefly recall fractional Sobolev spaces and the fractional Laplace operator to fix notation. In Section \ref{sec:problem} we define the geometrical setting of the problem, we propose a definition of weak solution and we state our main existence result. Section \ref{sec: time discr} is devoted to the proof via the time-discretization method, while Section \ref{sec: penalty} provides the proof via a Galerkin-penalization approximation. Appendix \ref{appendix} recalls a result for the problem in cylindrical domains that can be obtained as a simple adaptation of results in \cite{bonafini2021obstacle, bonafini2019variational}.

\section{Preliminaries and notations}\label{sec:pre not}

We denote by $\mathcal{L}^d$ the $d$-dimensional Lebesgue measure, and the integration with respect to the Lebesgue measure by $dx$. We adopt standard notations and definitions for Lebesgue spaces $L^p$ and Sobolev spaces $W^{n,p}$, $1 \leq p \leq +\infty$ and $n \in \N$. Given a Banach space $\mathcal{X}$, we denote by $\mathcal{X}^*$ its topological dual and by $\langle \varphi,x \rangle_{\mathcal{X}^\ast, \mathcal{X}}$ the duality pairing between $\varphi \in \mathcal{X}^*$ and $x \in \mathcal{X}$.

\medskip
\noindent
\textbf{Bochner spaces.} For $T > 0$, let $\{ \mathcal{X}_t \}_{t \in [0,T]}$ be a family of non-decreasing Banach spaces, i.e., $\mathcal{X}_s \subseteq \mathcal{X}_t$ for every $0 \leq s \leq t \leq T$, and denote by $\mathcal{X}:=\mathcal{X}_T$. With a slight abuse of notation, we say that a function $u$ belongs to $L^p(0,T; \mathcal{X}_t)$, $1 \leq p \leq +\infty$, if $u \in L^p(0,T; \mathcal{X})$ and $u(t) \in \mathcal{X}_t$ for a.e. $t \in [0,T]$. In particular, the map $t \mapsto \| u(t) \|_{\mathcal{X}}$ is in $L^p(0,T)$. 
We denote by $\cscal([0,T]; \mathcal{X})$ the subspace of $L^{\infty}(0,T; \mathcal{X})$ consisting of functions $u$ that are scalarly continuous, namely such that 
\[
\text{the map $t \mapsto \langle \varphi, u(t) \rangle_{\mathcal{X}^\ast, \mathcal{X}}$ is continuous on $[0,T]$ for every $\varphi \in \mathcal{X}^\ast$.}
\]
In particular, we recall the following result \cite[Lemma 8.1]{lionsmagenes}: if $\mathcal{X}$ and $\mathcal{Y}$ are two Banach spaces, $\mathcal{X}$ reflexive and $\mathcal{X}\subset \mathcal{Y}$ with continuous injection, then
\begin{equation*}
	L^{\infty}(0,T;\mathcal{X}) \cap \cscal([0,T]; \mathcal{Y}) = \cscal([0,T]; \mathcal{X}). 
\end{equation*}

\medskip
\noindent
\textbf{Fractional Sobolev spaces.} Throughout this paper, we fix $s > 0$ and $d,m \geq 1$. Let us consider the Schwartz space $\mathcal{S}(\R^d; \R^m)$, i.e., the space of rapidly decreasing $C^\infty$ functions from $\R^d$ to $\R^m$. For any $u \in \mathcal{S}(\R^d; \R^m)$ we define the Fourier transform $\mathcal{F}u$ of $u$ as
\[
\mathcal{F}u(\xi) := \frac{1}{(2\pi)^{d/2}} \int_{\R^d} \textup{e}^{-\textup{i} \xi \cdot x} u(x) \,dx, \quad \xi \in \R^d.
\]
The fractional Laplacian operator $(-\Delta)^s$ can be defined, up to constants, as 
\[
(-\Delta)^s \colon \mathcal{S}(\R^d;\R^m) \to L^2(\R^d;\R^m), \;u \mapsto (-\Delta)^s u := \mathcal{F}^{-1}(|\xi|^{2s}\mathcal{F}u). 
\]
Given $u, v \in L^2(\R^d;\R^m)$, we consider the bilinear form defined by
\[
[u, v]_{s} := \int_{\R^{d}} (-\Delta)^{s/2}u(x) \cdot (-\Delta)^{s/2}v(x) \, dx,
\]
and the corresponding Gagliardo seminorm
\begin{equation*}
[u]_s := \sqrt{[u,u]_s} = \|(-\Delta)^{s/2}u\|_{L^2(\R^d;\R^m)}.    
\end{equation*}
We define the fractional Sobolev space of order $s > 0$ as
\[
H^s(\R^d;\R^m) := \left\{ u \in L^2(\R^d;\R^m) \,:\, \int_{\R^d} (1+|\xi|^{2s})|\mathcal{F}u(\xi)|^2\,d\xi < +\infty \right\},
\]
and for every $u \in H^s(\R^d;\R^m)$ the $H^s$-norm of $u$ as
\begin{equation*}
\|u\|_s^2 = \|u\|_{L^2(\R^d;\R^m)}^2 + [u]_s^2.
\end{equation*}
It is well-known that $(H^s(\R^d;\R^m), \|\cdot\|_s)$ is a Hilbert space (see e.g. \cite{DiNezzaPalatucciValdinoci12, mclean2000strongly}).

Let $\Omega \subset \R^d$ be an open bounded set with Lipschitz boundary and define
\[
\tilde{H}^s(\Omega; \R^m) := \left\{ u \in H^s(\R^d;\R^m) \,:\, u = 0 \text{ a.e. in } \R^d\setminus \Omega \right\},
\]
endowed with the $H^s$-norm, and define its dual $H^{-s}(\Omega; \R^m) := (\tilde{H}^s(\Omega; \R^m))^*$. One can prove, see e.g. \cite{mclean2000strongly}, that $\tilde{H}^s(\Omega; \R^m)$ corresponds to the closure of $C^\infty_c(\Omega; \R^m)$ with respect to the $H^s$-norm. We have the continuous injections $\tilde{H}^s(\Omega; \R^m) \subset L^2(\Omega; \R^m) \subset H^{-s}(\Omega; \R^m)$. In particular, the duality product between $H^{-s}(\Omega; \R^m)$ and $\tilde{H}^s(\Omega; \R^m)$ satisfies
\begin{equation*}
    \langle f,u \rangle_{H^{-s}(\Omega; \R^m), \tilde{H}^s(\Omega; \R^m)} = \langle f,u \rangle_{L^2(\Omega; \R^m)} \quad \text{for every } f \in L^2(\Omega; \R^m) \text{ and } u \in \tilde{H}^s(\Omega; \R^m).
\end{equation*}

In the following, for notational convenience, we usually drop the target space $\R^m$ whenever clear from the context, and just write $L^2(\Omega)$, $\tilde{H}^s(\Omega)$ and ${H}^{-s}(\Omega)$. Moreover, each function $u \in L^2(\Omega)$ is naturally interpreted as a function of $L^2(\R^d)$ considering it extended to zero outside of $\Omega$.

\section{Formulation of the problem and main result}\label{sec:problem}
In this section, we introduce the fractional semilinear wave equation on non-cylindrical domains, we provide a definition of weak solution and we state our main result.

Let $T > 0$ be a fixed final time. Consider a family $\{\Omega_t\}_{t\in [0,T]} \subset \R^d$ of nonempty open bounded sets with Lipschitz boundary and define the domain $\mathcal{O} \subset [0,T] \times \R^d$ by
\begin{equation*}
 \mathcal{O}:= \bigcup_{t\in(0,T)}\{t\}\times \Omega_t.
\end{equation*}
We will assume the following structural conditions:
\begin{equation}\label{hyp on h}
\begin{split}
&\raisebox{0.25ex}{\tiny$\bullet$} \, \, \text{ the family } \{\Omega_t\}_{t \in [0,T]} \text{ is non-decreasing, i.e., } \Omega_{s} \subset \Omega_t \text{ for every } 0 \leq s \leq t \leq T,
\\
&\raisebox{0.25ex}{\tiny$\bullet$} \, \,
\mathcal{O} \subset \R \times \R^d \text{ is an open set with Lipschitz boundary}.
\end{split}
\end{equation}
For a function $u \colon [0,T] \times \R^d \to \R^m$, let us consider the problem
\begin{equation}\label{eq:u}
	\begin{cases}
		\ddot{u}+(-\Delta)^s u  + \nabla W(u) = f &\text{in }\mathcal{O},\\
		u=0 &\text{in }((0,T)\times \R^d) \setminus \mathcal{O},\\
		u(0,x)=u_0(x) & \text{for } x \in \Omega_0,\\
		\dot{u}(0,x)=v_0(x) &\text{for }  x \in \Omega_0,
	\end{cases}
\end{equation}
where: 
\begin{subequations}
\begin{align}
    \raisebox{0.25ex}{\tiny$\bullet$} \, \, & W \in C^1(\R^m) \, \text{is a non-negative potential with Lipschitz continuous gradient, } \nonumber 
    \\
    &\text{i.e., $\exists \mathrm{K} > 0$ such that } |\nabla W(x) - \nabla W(y)| \leq \mathrm{K} |x-y| \text{ for all } x,y \in \R^m,
    \label{eq: nabla W}
    \\
    \raisebox{0.25ex}{\tiny$\bullet$} \, \, & f\in L^\infty(0,T; L^2(\R^d)) \text{ with } \supp f = \overline{\{(t,x) \in [0,T] \times \R^d \mid f(t,x) \neq 0\}}\subset \mathcal{O}, \label{eq:f}
    \\
    \raisebox{0.25ex}{\tiny$\bullet$} \, \, &u_0\in \tilde{H}^s(\Omega_0) \text{ and }
			v_0\in L^2(\Omega_0). \label{eq: u_0 v_0}
\end{align}
\end{subequations} 

\begin{remark}
    Problem \eqref{eq:u} consists of a fractional semilinear wave equation in the non-cylindrical domain $\mathcal{O}$ with forcing term $f$.  Note that in \eqref{eq:u} the homogeneous Dirichlet condition is prescribed in $((0,T)\times \R^d) \setminus \mathcal{O}$ and not simply on the parabolic boundary of $\mathcal{O}$, consistently with the non-local character of the operator $(-\Delta)^s$. 
\end{remark}

\begin{remark}\label{property of W}
    Since the PDE in \eqref{eq:u} depends on $\nabla W$, the set of solutions of the problem is invariant under translation of the potential $W$. Hence, without loss of generality, we can always assume 
    \begin{equation*}
        W(0) = 0.
    \end{equation*}
    In particular, condition \eqref{eq: nabla W} implies at most linear growth of $\nabla W$: there exists $C_W > 0$ such that
    \begin{equation}\label{ineq for nabla W(y)}
    |\nabla W(y)| \leq C_{W} (1 + |y|) \quad \text{for all } y \in \R^m.
    \end{equation}
    Furthermore, one can prove that there exists two positive constants $C_1,C_2$ such that the following holds (see also \cite[Prop. 8 (iii)]{bonafini2021obstacle}):
    \begin{equation}\label{ineq for W(x)-W(y)}
        |W(x)-W(y)| \leq \left(C_1 (|x|+|y|) + C_2 \right) |x-y| \quad \text{for all } x,y \in \R^m.
    \end{equation}
\end{remark}

We define a weak solution to problem \eqref{eq:u} as follows.

\begin{definition}[Weak solution]\label{defweaksolu}
	We say that a function $u \colon [0, T] \times \R^{d}\to \R^m$ is a weak solution to problem \eqref{eq:u} if the following hold:
	\begin{enumerate}
		\item[(i)] $u \in L^\infty(0,T; \tilde{H}^s(\Omega_t))$ and $\dot{u} \in L^\infty(0,T; L^2(\Omega_t))$;
  
        \item[(ii)] $u \in C^0([0,T]; L^2(\Omega_T))$ and $\dot{u} \in C^0([0,T]; H^{-s}(\Omega_0))$, $u(0)=u_0$ and $\dot{u}(0) = v_0$;
  
		\item[(iii)]  $u$ satisfies
		\begin{equation}\label{def eq u weak sol}
			\begin{aligned}
				-\int_{0}^{T} \langle \dot{u}(t), \dot{\varphi}(t) \rangle_{L^2(\Omega_T)} \, dt +\int_{0}^{T}[u(t), \varphi(t)]_s \,dt&
            	\\+\int_{0}^{T} \scal{\nabla W(u(t))}{\varphi(t)}_{L^2(\Omega_{T})} \,dt
				&= \int_{0}^{T} \langle f(t), \varphi(t) \rangle_{L^2(\Omega_T)}\, dt,
			\end{aligned}
		\end{equation}
		for every $\varphi \in L^1(0,T; \tilde{H}^s(\Omega_T)) \cap W^{1,1}(0,T; L^2(\Omega_T))$ with $\supp \varphi \subset \mathcal{O}$.
	\end{enumerate}
\end{definition}

Finally, we introduce the energy definition for a weak solution to problem \eqref{eq:u}.
\begin{definition}\label{def energy}
	Let $u$ be a weak solution to \eqref{eq:u}. We define the energy of $u$ as
	\begin{equation*}
    E(u(t)) := \frac12 \|\dot{u}(t)\|^{2}_{L^{2}(\Omega_t)}+\frac12 [u(t)]_{s}^2+\|W(u(t))\|_{L^{1}(\Omega_t)} \quad \text{for every }t \in [0,T],
    \end{equation*}
    and we set
    \begin{equation*}
        E_0 := E(u(0)) =  \frac12 \norm{v_0}^2_{L^2(\Omega_{0})} +\frac12 [u_0]_{s}^2+\|W(u_0)\|_{L^1(\Omega_{0})}.
    \end{equation*}
\end{definition}

Our main result reads as follows.
\begin{theorem}\label{thm main result}
There exists a weak solution $u$ to problem \eqref{eq:u} in the sense of Definition \ref{defweaksolu}. Moreover, the following energy inequality holds:
\begin{equation}\label{eq: en ineq for u}
	E(u(t)) 
    \leq
    E_0 + \int_{0}^{t}\scal{f(\tau)}{\dot{u}(\tau)}_{L^2(\Omega_\tau)}\, d\tau \qquad \text{ for every $t\in [0,T]$}.
\end{equation}
\end{theorem}

\begin{remark}
    In this paper we focus only on the existence of weak solutions to problem \eqref{eq:u}.  We point out that, in \cite{MR3884619}, a problem closely related to \eqref{eq:u} has been investigated in the domain $(0,T) \times \mathbb{R}^d$. There, the authors establish the existence of solutions for a non-linear non-local dispersive evolution equation featuring singular and non-Lipschitz interactions confined to a finite horizon. Their proof relies on a regularization procedure that yields solutions that take values in a suitable fractional-type space. Furthermore, the authors are able to prove wellposedness of their model. We expect that establishing uniqueness in the setting of time-dependent domains requires stronger assumptions on the family $\{\Omega_t\}_{t\in [0,T]}$, which goes beyond the scope of this work.
\end{remark}

\section{Time discretization method}\label{sec: time discr}
In this section, we provide a proof of Theorem \ref{thm main result} based on a time-discretization method in the spirit of \cite{CalvoNovagaOrlandi} and \cite{LazzaroniMolinaroloRivaSolombrino} (see also \cite{LAZZARONI2022112822} for the radial setting). The idea is the following: for each $n \in \N$,
we approximate $\mathcal{O}$ from the interior by
\begin{equation*}
    \bigcup_{k=1}^n (t_{k-1}^{n}, t_{k}^{n}) \times \Omega_{t_{k-1}^{n}} \subseteq \mathcal{O},
\end{equation*}
where $\{t^n_k\}_{k=0}^{n}$ is a uniform partition of the interval $[0,T]$. We seek, at each step $k = 1, \dots, n$, a solution $u^n_k$ of the PDE
\[
\ddot{u}+(-\Delta)^s u  + \nabla W(u) = f \quad \text{in } (t^n_{k-1}, t^n_{k}) \times \Omega_{t^n_{k-1}}  
\]
with initial data determined through the solution at the previous step $k-1$. An existence result in the cylindrical case was proved in \cite{bonafini2021obstacle} for the homogeneous case (i.e., without the source term $f$) and, for convenience, we extend such existence result to the non-homogeneous case in Appendix \ref{appendix}. We then prove that the piecewise constant in time approximants $\{u^n\}_n$ built from the families $\{u^n_k\}_{n,k}$ yield, in the limit, a weak solution to \eqref{eq:u} in the sense of Definition  \ref{defweaksolu}.

\begin{proof}[Proof of Theorem \ref{thm main result}]
$\,$

\medskip
\noindent
\emph{Step 1. Discrete approximants and energy estimates.}
For every $n \in \N$, let $\{t^n_k\}_{k=0}^{n}$ be the uniform partition of the interval $[0,T]$ given by
\begin{equation*}
t_k^n := \frac{kT}{n} \quad \text{for } k =0, 1, \dots, n.
\end{equation*}
Upon defining $u^n_{0}(t,x) = u_0(x) + tv_0(x)$, we inductively define
\[
u^n_k \colon [t^n_{k-1},t^n_{k}] \times \R^d \to \R^m \quad \text{for } k=1,\dots,n,
\]
as a weak solution to the fractional semilinear wave equation
\begin{equation}\label{eq for u^n_k}
\begin{cases}
\ddot u + (-\Delta)^s u + \nabla W (u) =f &\text{in }(t^n_{k-1},t^n_{k})\times\Omega_{t^n_{k-1}},\\
u = 0 &\text{in } (t^n_{k-1},t^n_{k}) \times (\R^d \setminus \Omega_{t^n_{k-1}}),\\
u(t^n_{k-1}, x)=u^n_{k-1}(t^n_{k-1},x) & \text{for } x \in \Omega_{t^n_{k-1}},\\
\dot{u}(t^n_{k-1}, x)=\dot{u}^n_{k-1}(t^n_{k-1}, x)& \text{for } x \in \Omega_{t^n_{k-1}}.
\end{cases}
\end{equation}
Indeed, since we are considering a fixed cylindrical domain $(t^n_{k-1},t^n_{k})\times\Omega_{t^n_{k-1}}$, Theorem \ref{app:thmexistence} in Appendix \ref{appendix} guarantees the existence of a weak solution $u_k^n$ to \eqref{eq for u^n_k} such that
\begin{subequations}\label{eq:uknproperties}
\begin{equation}
\begin{aligned}\label{eq:spacesukn}
	u^n_k &\in L^\infty(t^n_{k-1}, t^n_k; \tilde{H}^s(\Omega_{t^n_{k-1}})) \cap C^0([t^n_{k-1}, t^n_k]; L^2(\Omega_{t^n_{k-1}})) \cap \cscal([t^n_{k-1}, t^n_k]; \tilde{H}^s(\Omega_{t^n_{k-1}})), \\
	\dot{u}^n_k &\in L^{\infty}(t^n_{k-1}, t^n_k;L^2(\Omega_{t^n_{k-1}})) \cap C^0([t^n_{k-1}, t^n_k]; H^{-s}(\Omega_{t^n_{k-1}})) \cap \cscal([t^n_{k-1}, t^n_k]; L^{2}(\Omega_{t^n_{k-1}})), \\
	\ddot{u}^n_k &\in L^{\infty}(t^n_{k-1}, t^n_k;H^{-s}(\Omega_{t^n_{k-1}})),
\end{aligned}
\end{equation}
with initial conditions satisfied as
\begin{equation}\label{eq:ICukn}
	\begin{aligned}
	u^n_k(t^n_{k-1}) &= u^n_{k-1}(t^n_{k-1}) \text{ in the sense of } \tilde{H}^s(\Omega_{t^n_{k-1}}), \\ \dot{u}^n_k(t^n_{k-1}) &= \dot{u}^n_{k-1}(t^n_{k-1}) \text{ in the sense of } L^2(\Omega_{t^n_{k-1}}),
	\end{aligned}
\end{equation}
and with $u_k^n$ satisfying
\begin{equation}\label{eq:weakunk}
	\begin{split}
		&-\int_{t^n_{k-1}}^{t^n_k} \scal{\dot{u}^n_k(t)}{\dot{\varphi}(t)}_{L^2(\Omega_{t^n_{k-1}})} dt + \int_{t^n_{k-1}}^{t^n_k} [ u^n_k(t), \varphi(t) ]_{s} \, dt +\int_{t^n_{k-1}}^{t^n_k} \scal{ \nabla W(u^n_k(t))}{\varphi(t)}_{L^2(\Omega_{t^n_{k-1}})} \,dt \\
		&= \scal{\dot{u}^n_k(t^n_{k-1})}{\varphi(t^n_{k-1})}_{L^2(\Omega_{t^n_{k-1}})} -\scal{\dot{u}^n_k(t^n_k)}{\varphi(t^n_k)}_{L^2(\Omega_{t^n_{k-1}})} + \int_{t^n_{k-1}}^{t^n_k} \scal{f(t)}{\varphi(t)}_{L^2(\Omega_{t^n_{k-1}})}\,dt \\
		&\text{for every $\varphi \in L^1(t^n_{k-1}, t^n_k; \tilde{H}^s(\Omega_{t^n_{k-1}})) \cap W^{1,1}(t^n_{k-1}, t^n_k; L^2(\Omega_{t^n_{k-1}}))$.}
	\end{split}
\end{equation}
In particular,
\begin{equation}\label{ineq: en bal u^n_k}
	\begin{aligned}
		E(u_k^n(t)) &= \frac12 \norm{\dot{u}^n_k(t)}^2_{L^2(\Omega_{t^n_{k-1}})}+\frac12 [u^n_k(t)]_{s}^2 + \|W(u^n_k(t))\|_{L^1(\Omega_{t^n_{k-1}})}
		\\
		& \leq \frac 12\norm{\dot{u}^n_{k}(t^n_{k-1})}^2_{L^2(\Omega_{t^n_{k-1}})}+ \frac12 [{u}^n_{k}(t^n_{k-1})]_{s}^2 +\|W(u^n_{k}(t^n_{k-1}))\|_{L^1(\Omega_{t^n_{k-1}})}
		\\
		& \quad +\int_{t^n_{k-1}}^{t}\scal{f(\tau)}{\dot{u}^n_k(\tau)}_{L^2(\Omega_{t^n_{k-1}})}
		\,d\tau \qquad \text{for every $t\in [t^n_{k-1}, t^n_k]$}. 
	\end{aligned}
\end{equation}
\end{subequations}

\medskip
Given the sequence $\{u^n_k\}_{k=1}^n$, we define the approximate solution $u^n: [0,T] \times \R^d \to \R^m $ as
\begin{equation}\label{eq: def u^n}
u^n(t,x) :=
\begin{cases}
    u^n_k(t,x) &\text{if } t \in  [t^n_{k-1}, t^n_k) \text{ for some } k \in \{1, \dots, n\}, 
    \\
    u^n_n(T,x) &\text{if } t = T.
\end{cases}
\end{equation}
From \eqref{eq:uknproperties}, we deduce that
\begin{subequations}\label{eq:unproperties}
	\begin{equation}
		\begin{aligned}
			u^n \in L^\infty(0,T;\tilde{H}^s(\Omega_{T})), \text{ and }
			\dot{u}^n \in L^\infty(0,T;L^2(\Omega_{T})),
		\end{aligned}
	\end{equation}
	with initial time conditions
	\begin{equation}\label{eq:ICun}
		\begin{aligned}
			u^n(0)&=u^n_{1}(0) = u_0 \quad\text{in the sense of } \tilde{H}^s(\Omega_0),
			\\
			\dot{u}^n(0)&= \dot{u}^n_{1}(0) = v_0 \quad\text{in the sense of } L^2(\Omega_0).
		\end{aligned}
	\end{equation}
	Summing \eqref{eq:weakunk} for $k = 1, \dots, n$, and using that $\dot{u}^n_k(t^n_{k-1})=\dot{u}^n_{k-1}(t^n_{k-1})$ to simplify the boundary terms, we obtain that
	\begin{equation}\label{eq: weak u^n}
		\begin{aligned}
			&-\int_{0}^{T}\scal{\dot{u}^n(t)}{\dot{\varphi}(t)}_{L^2(\Omega_{T})} \, dt +\int_{0}^{T}[u^n(t), \varphi(t)]_s \,dt +\int_{0}^{T} \scal{\nabla W(u^n(t))}{\varphi(t)}_{L^2(\Omega_{T})} \,dt \\
			&= \int_{0}^{T}\scal{f(t)}{\varphi(t)}_{L^2(\Omega_{T})} \,dt \\
			& \text{for every $\varphi\in C^{\infty}_{c}((0,T) \times \Omega_T)$ with }\supp \varphi \subseteq \bigcup_{k=1}^{n} [t^n_{k-1}, t^n_k] \times \Omega_{t^n_{k-1}}.
		\end{aligned}
	\end{equation}
	Furthermore, the energy estimate \eqref{ineq: en bal u^n_k} provides that
	\begin{equation}\label{ineq: en bal u^n 2}
		\begin{aligned}
			E(u^n(t)) &=\frac12 \norm{\dot{u}^n(t)}^2_{L^2(\Omega_{t})}+\frac12 [u^n(t)]_{s}^2+\|W(u^n(t))\|_{L^1(\Omega_{t})} \\
			&\leq \frac12 \norm{v_0}^2_{L^2(\Omega_{0})} +\frac12 [u_0]_{s}^2+\|W(u_0)\|_{L^1(\Omega_{0})} + \int_{0}^{t} \scal{f(\tau)}{\dot{u}^n(\tau)}_{L^2(\Omega_{\tau})} \,d\tau \\
			& =E_0 + \int_{0}^{t} \scal{f(\tau)}{\dot{u}^n(\tau)}_{L^2(\Omega_{\tau})} \,d\tau  \qquad \text{for every $t\in [0, T]$}. 
		\end{aligned}
	\end{equation}
\end{subequations}
Indeed, let $t\in[0,T]$ and let $k \in \{1,\dots,n\}$ such that $t^n_{k-1} < t \leq t^n_k$. Then, we estimate
\begin{equation}\label{ineq: en bal u^n}
	\begin{aligned}
		E(u^n(t)) &=\frac12 \norm{\dot{u}^n(t)}^2_{L^2(\Omega_{t})}+\frac12 [u^n(t)]_{s}^2+\|W(u^n(t))\|_{L^1(\Omega_{t})} \\
		&= \frac12 \norm{\dot{u}^n_k(t)}^2_{L^2(\Omega_{t^n_{k-1}})}+\frac12 [u^n_k(t)]_{s}^2 + \|W(u^n_k(t))\|_{L^1(\Omega_{t^n_{k-1}})} + \underbrace{\|W(0)\|_{L^1(\Omega_t \setminus\Omega_{t^n_{k-1}})}}_{=0}
		\\
		&\overset{\eqref{eq:ICukn}, \eqref{ineq: en bal u^n_k}}{\leq} \underbrace{E(u^n_{k-1}(t^n_{k-1}))}_{=E(u^n(t_{k-1}^n))} +\int_{t^n_{k-1}}^{t}\scal{f(\tau)}{\dot{u}^n_k(\tau)}_{L^2(\Omega_{\tau})}
		\,d\tau \qquad \text{for every } t \in [t_{k-1}^n, t_k^n],
	\end{aligned}
\end{equation}
where for every $\tau \in (t^n_{k-1}, t^n_k]$ we can interchange the spaces $L^2(\Omega_{\tau})$ and $L^2(\Omega_{t^n_{k-1}})$ because we have $\Omega_{t^n_{k-1}} \subseteq \Omega_\tau$ and $\dot{u}^n_k(\tau) = 0$ in $\Omega_\tau \setminus \Omega_{t^n_{k-1}}$. Hence, \eqref{ineq: en bal u^n 2} follows by iterating \eqref{ineq: en bal u^n}.

\medskip
\noindent
\emph{Step 2. Pre-compactness results.} It follows from \eqref{ineq: en bal u^n 2} that for every $t \in [0,T]$ we have
\begin{equation}\label{eq:step0gronwall}
\begin{aligned}
	&\frac12 \norm{\dot{u}^n(t)}^2_{L^2(\Omega_{T})} \leq E_0 + \int_{0}^{t} \scal{f(\tau)}{\dot{u}^n(\tau)}_{L^2(\Omega_{T})} \,d\tau,
	\\
	& \frac12 [u^n(t)]_{s}^2 \leq E_0 + \int_{0}^{t} \scal{f(\tau)}{\dot{u}^n(\tau)}_{L^2(\Omega_{T})} \,d\tau.
\end{aligned}
\end{equation}

We now proceed with a classical Gr\"onwall's argument. By Cauchy inequality, we infer that
\begin{align*}
    \frac12 \norm{\dot{u}^n(t)}^2_{L^2(\Omega_{T})} &\leq \underbrace{E_0 + \frac{1}{2}\int_{0}^{T} \norm{f(\tau)}^2_{L^2(\Omega_{T})} \,d\tau}_{=:\,C_0} + \frac{1}{2}\int_{0}^{t} \norm{\dot{u}^n(\tau)}^2_{L^2(\Omega_{T})} \,d\tau
    \\
    &=C_0 +  \frac{1}{2}\int_{0}^{t} \norm{\dot{u}^n(\tau)}^2_{L^2(\Omega_{T})} \,d\tau. 
\end{align*}
By Gr\"onwall's inequality (see, e.g., \cite{L.C.Evans}), we obtain
\begin{equation*}
    \frac12 \norm{\dot{u}^n(t)}^2_{L^2(\Omega_{T})} \leq  C_0(1 + T\textup{e}^{T}) \quad \text{for every } t \in [0,T].
\end{equation*}
Hence, from \eqref{eq:step0gronwall} we get
\begin{equation}\label{ineq: gron argument}
    \sup_{t \in [0,T]} \left( \norm{\dot{u}^n(t)}^2_{L^2(\Omega_{T})} + [u^n(t)]_{s}^2  \right) \leq C,
\end{equation}
for a suitable positive constant $C$ independent of $n$. In other words, $\{u^n\}_n$ is equi-bounded in $L^\infty(0,T; \tilde{H}^s(\Omega_T))$ and $\{\dot{u}^n\}_n$ in equi-bounded in $L^\infty(0,T; L^2(\Omega_T))$.

Furthermore, $\{\ddot{u}^n\}_n$ is equi-bounded in $L^\infty(0,T; H^{-s}(\Omega_0))$. Indeed, we first note that by \eqref{eq: def u^n} and \eqref{eq:spacesukn}, we have that every $\ddot{u}^n \in L^\infty(0,T; H^{-s}(\Omega_0))$ for every $n \in \N$. In particular, we aim to establish the following inequality:
\begin{equation}\label{eq: ddot u^n equibound}
    \left| \int_{0}^{T} \langle \ddot{u}^n(t), \varphi(t) \rangle_{H^{-s}(\Omega_{0}), \tilde{H}^s(\Omega_{0})} \,dt \right| \leq D \norm{\varphi}_{L^{1}(0,T;\tilde{H}^s(\Omega_0))} \quad \text{for every } \varphi \in L^{1}(0,T;\tilde{H}^s(\Omega_0)), 
\end{equation}
for a suitable positive constant $D$ independent of $n$ and $t$. Hence, consider a smooth function $\varphi \in C_c^{\infty}((0,T)\times \Omega_0)$. Since $\varphi$ is an admissible choice as test function in \eqref{eq: weak u^n}, we obtain that
\begin{equation}\label{eq: ddot u^n}
		\begin{aligned}
        & \int_{0}^{T} \langle \ddot{u}^n(t), \varphi(t) \rangle_{H^{-s}(\Omega_{0}), \tilde{H}^s(\Omega_{0})} \,dt 
        = -\int_{0}^{T}\scal{\dot{u}^n(t)}{\dot{\varphi}(t)}_{L^2(\Omega_{0})} \, dt 
        \\
        &= -\int_{0}^{T}[u^n(t), \varphi(t)]_s \,dt 
        - \int_{0}^{T} \scal{\nabla W(u^n(t))}{\varphi(t)}_{L^2(\Omega_{0})} \,dt
		+\int_{0}^{T}\scal{f(t)}{\varphi(t)}_{L^2(\Omega_{0})} \,dt.
		\end{aligned}
\end{equation}
Now we estimate the three terms in the second line above.

\smallskip
For the first term, one has the inequality
\begin{equation*}
\begin{split}
    \left|\int_{0}^{T}[u(t), \varphi(t)]_s \,dt\right|
    &\leq  \int_{0}^{T} \left| [u^n(t), \varphi(t)]_s \right| \,dt 
    \\
    \text{\footnotesize (by Cauchy Schwarz for $[\cdot,\cdot]_s$)} \qquad &\leq  \int_{0}^{T} [u^n(t)]_s \, [\varphi(t)]_s \,dt 
    \\
    \text{\footnotesize (by the inequality $[\cdot]_s \leq \|\cdot\|_s$)} \qquad & \leq \underbrace{\sqrt{C}}_{D_1} \norm{\varphi}_{L^1(0,T;\tilde{H}^s(\Omega_0))} ,
\end{split}
\end{equation*}
where we used \eqref{ineq: gron argument} to get $[u^n(t)]_s \leq \sqrt{C}$ for every $t \in [0,T]$.

\smallskip
For the second term, it suffices to consider the estimate
\begin{equation*}
\begin{split}
    \left| \int_{0}^{T} \scal{\nabla W(u^n(t))}{\varphi(t)}_{L^2(\Omega_{0})} \,dt 
    \right| 
    & \leq \int_{0}^{T}  \left| \scal{\nabla W(u^n(t))}{\varphi(t)}_{L^2(\Omega_{0})} \right| \,dt 
    \\
    \text{\footnotesize (by Cauchy Schwarz in $L^2(\Omega_0)$)} \quad  &\leq \int_{0}^{T}  \norm{\nabla W(u^n(t))}_{L^2(\Omega_{0})} \norm{\varphi(t)}_{L^2(\Omega_{0})}  \,dt 
    \\
    & \leq \norm{\nabla W(u^n)}_{L^{\infty}(0,T;L^2(\Omega_{0}))} \int_{0}^{T} \norm{\varphi}_{L^2(\Omega_{0})}  \,dt 
    \\
    & \leq D_2 \norm{\varphi}_{L^{1}(0,T;L^2(\Omega_{0}))}, 
\end{split}
\end{equation*}
where the constant $D_2$ is given by
\begin{align*}
    \esssup_{t \in (0,T)} \norm{\nabla W(u^n(t))}_{L^2(\Omega_{0})}  \,dt & \overset{\eqref{ineq for nabla W(y)}}{\leq}  \esssup_{t \in (0,T)} \left( C_W \int_{\Omega_0} (1+|u^n(t,x)|)^2 \,dx \right)^\frac{1}{2}
    \\
    \text{\footnotesize (using $(a+b)^2 \leq 2a^2+2b^2$)} \quad
    & \leq \underbrace{  \sqrt{2C_W} \left( \mathcal{L}^d(\Omega_0) + \norm{u^n}^2_{L^\infty(0,T;L^2(\Omega_T))} \right)^\frac{1}{2}}_{D_2}
\end{align*}
where the uniform bound of $u^n$ in $L^\infty(0,T;L^2(\Omega_T))$ follows again from \eqref{ineq: gron argument}.

\smallskip
For the third term, one has (again by Cauchy Schwarz in $L^2(\Omega_0)$)
\begin{equation*}
\begin{split}
    \left|\int_{0}^{T} \langle f(t), \varphi(t) \rangle_{L^2(\Omega_0)}\, dt \right| &\leq \int_{0}^{T} \norm{f(t)}_{L^2(\Omega_0)} \norm{\varphi(t)}_{L^2(\Omega_0)} \, dt
    \\
    &\overset{\eqref{eq:f}}{\leq} \underbrace{\norm{f}_{L^{\infty}(0,T;L^2(\Omega_{0}))} }_{D_3} \norm{\varphi}_{L^{1}(0,T;L^2(\Omega_{0}))}. 
\end{split}
\end{equation*}

\smallskip
\noindent
Hence, the estimates above yield the existence of positive constants $D_1,D_2, D_3$ such that 
\begin{equation}\label{ineq: D_1 D_2 D_3}
    \begin{aligned}
    \left|\int_{0}^{T}[u^n(t), \varphi(t)]_s \,dt\right| & \leq D_1 \norm{\varphi}_{L^1(0,T; \tilde{H}^s(\Omega_0))},
    \\
    \left| \int_{0}^{T} \scal{\nabla W(u^n(t))}{\varphi(t)}_{L^2(\Omega_{0})} \,dt \right|  & \leq D_2 \norm{\varphi}_{L^1(0,T; \tilde{H}^s(\Omega_0))},
    \\
    \left| \int_{0}^{T}\scal{f(t)}{\varphi(t)}_{L^2(\Omega_{0})} \,dt \right| & \leq D_3 \norm{\varphi}_{L^1(0,T; \tilde{H}^s(\Omega_0))}.
\end{aligned}
\end{equation}
By the density of $C_c^{\infty}((0,T)\times \Omega_0)$ in $L^1(0,T; \tilde{H}^s(\Omega_0))$, \eqref{eq: ddot u^n} and \eqref{ineq: D_1 D_2 D_3} imply the validity of \eqref{eq: ddot u^n equibound}. Hence the equi-boundedness of $\{\ddot{u}^n\}_n$ in $L^\infty(0,T; H^{-s}(\Omega_0))$ holds.

\medskip
\noindent
\emph{Step 3. Convergence of the approximants.}
By \eqref{ineq: gron argument}, there exist $u \in L^\infty(0,T; \tilde{H}^s(\Omega_T))$ and $v \in L^\infty(0,T; L^2(\Omega_T))$  such that
\begin{equation}\label{eq: convergence u^n to u}
    \begin{aligned}
    & u^n \rightharpoonup^* u \text{ in } L^\infty(0,T; \tilde{H}^s(\Omega_T)),
    \\
    & \dot{u}^n \rightharpoonup^* v \text{ in } L^\infty(0,T; L^2(\Omega_T)).
\end{aligned}
\end{equation}
It is standard to prove that $v = \dot{u}$ in $L^\infty(0,T;L^2(\Omega_T))$ and, since 
\begin{equation*}
    \supp u^n \subseteq \bigcup_{k=1}^{n} [t^n_{k-1}, t^n_k] \times \Omega_{t^n_{k-1}} \subset \mathcal{O},    
\end{equation*} 
we conclude that $u \in L^\infty(0,T; \tilde{H}^s(\Omega_t))$ and $\dot{u} \in L^\infty(0,T; L^2(\Omega_t))$ (hence (i) in Definition \ref{defweaksolu} holds). Moreover, by the equi-boundedness of $\{\ddot{u}^n\}_n$ in $L^\infty(0,T; H^{-s}(\Omega_0))$ together with an Ascoli--Arzelà argument for the family $\{\dot u^n\}_n$ in $C^0([0,T]; H^{-s}(\Omega_0))$ and for the family $\{u^n\}_n$ in $C^0([0,T]; L^2(\Omega_T))$, we conclude that
\begin{equation}\label{eq: convergence u^n to u in C0}
\begin{aligned}
    & u^n \to u \text{ in } C^0([0,T]; L^2(\Omega_T)),
    \\
    & \dot{u}^n \to \dot{u} \text{ in } C^0([0,T]; H^{-s}(\Omega_0)).
\end{aligned}
\end{equation}
Furthermore, by \cite[Lemma 8.1]{lionsmagenes} we have that
\begin{equation*}
    L^\infty(0,T;\tilde{H}^s(\Omega_T)) \cap C^0([0,T]; L^2(\Omega_T)) \subseteq \cscal([0,T]; \tilde{H}^s(\Omega_T)), 
\end{equation*}
hence we also obtain $u \in \cscal([0,T]; \tilde{H}^s(\Omega_T))$. In a similar way, one observes that
\begin{equation*}
    L^\infty(0,T;L^2(\Omega_0)) \cap C^0([0,T]; H^{-s}(\Omega_0)) \subseteq \cscal([0,T]; L^2(\Omega_0)), 
\end{equation*}
hence $\dot{u} \in \cscal([0,T]; L^2(\Omega_{0}))$. 
In particular, using \eqref{eq:ICun}, the convergence provided by \eqref{eq: convergence u^n to u in C0} and the embeddings above, one obtains that (ii) in Definition \ref{defweaksolu} holds.

In order to conclude that $u$ is indeed a weak solution, we are left to prove that $u$ satisfies \eqref{def eq u weak sol}. We reason by density, hence proving the weak formulation for smooth functions. To this aim, we first observe that given a test function 
\begin{equation}\label{eq varphi smooth test}
    \varphi \in C^{\infty}_{c}((0,T)\times \Omega_T) \text{ with } \supp \varphi \subset \mathcal{O},
\end{equation}
then there exists $\bar{n} \in \N$ sufficiently large such that for all $n \geq \bar{n}$ it holds
\begin{equation*}
    \supp \varphi \subset \bigcup_{k=1}^{n} [t^{n}_{k-1}, t^{n}_k] \times \Omega_{t^{n}_{k-1}} \subset \mathcal{O}
\end{equation*}
(for a proof of the Hausdorff convergence of the increasing union of cylindrical domains to the set $\mathcal{O}$, see, for instance, \cite[Lem. 3.4]{CalvoNovagaOrlandi}).
Hence, we can take $\varphi^n = \varphi$ for all $n \geq \bar{n}$ in \eqref{eq: weak u^n}, namely 
\begin{equation*}
	\varphi \in C^{\infty}_c((0,T)\times \Omega_T), \text{ with } \supp \varphi \subseteq \bigcup_{k=1}^{n} [t^n_{k-1}, t^n_k] \times \Omega_{t^n_{k-1}},
\end{equation*} 
and then \eqref{def eq u weak sol} holds for every $\varphi$ as in \eqref{eq varphi smooth test} passing to the limit in \eqref{eq: weak u^n}: indeed, we have that
\begin{subequations}\label{subeq: lim equation}
    \begin{align}
        \label{subeq: lim equation 1} \lim_{n \to \infty} \int_{0}^{T}\scal{\dot{u}^n(t)}{\dot{\varphi}(t)}_{L^2(\Omega_{T})} \, dt &= \int_{0}^{T}\scal{\dot{u}(t)}{\dot{\varphi}(t)}_{L^2(\Omega_{T})} \, dt
        \\
        \label{subeq: lim equation 2} \lim_{n \to +\infty}\int_{0}^{T}[u^n(t), \varphi(t)]_s \,dt &= \int_{0}^{T}[u(t), \varphi(t)]_s \,dt,
        \\
        \label{subeq: lim equation 3} \lim_{n \to +\infty}\int_{0}^{T} \scal{\nabla W(u^n(t))}{\varphi(t)}_{L^2(\Omega_{T})} \,dt &= \int_{0}^{T} \scal{\nabla W(u(t))}{\varphi(t)}_{L^2(\Omega_{T})} \,dt.  
    \end{align}
\end{subequations}
As before we proceed by steps to verify \eqref{subeq: lim equation}.

\smallskip
For the first limit \eqref{subeq: lim equation 1} it suffices to consider the estimate
\begin{equation*}
\begin{split}
    \left| \int_{0}^{T} \langle \dot{u}^n(t), \dot{\varphi}(t) \rangle_{L^2(\Omega_T)} \, dt 
    \right.&\left.- 
    \int_{0}^{T} \langle \dot{u}(t), \dot{\varphi}(t) \rangle_{L^2(\Omega_T)} \, dt \right|
    \\
    &\leq \left| \int_{0}^{T} \langle \dot{u}^n(t) -\dot{u}(t), \dot{\varphi}(t) \rangle_{L^2(\Omega_T)} \, dt \right|, 
\end{split}
\end{equation*}
and the convergence $\dot{u}^n \rightharpoonup^* \dot{u}$ in $L^\infty(0,T; L^2(\Omega_T))$ (using $\dot{\varphi} \in C^{\infty}_c((0,T)\times \Omega_T) \subset  L^1(0,T;L^2(\Omega_T))$ as test function for the convergence in \eqref{eq: convergence u^n to u}).

\smallskip
For the second limit \eqref{subeq: lim equation 2}, we note that the regularity of $\varphi$ implies 
\[(-\Delta)^s \varphi \in L^1(0,T; H^{-s}(\Omega_T)),\] 
hence we get
\begin{align*}
    \left|\int_{0}^{T}[u^n(t), \varphi(t)]_s \,dt - \int_{0}^{T}[u(t), \varphi(t)]_s \,dt
    \right| \leq \left| \int_{0}^{T}[u^n(t) - u(t), \varphi(t)]_s \,dt  \right| &
    \\
    \leq \left| \int_{0}^{T} \scal{(- \Delta)^s\varphi(t)}{u^n(t) - u(t)}_{{H^{-s}(\Omega_{T}), \tilde{H}^s(\Omega_{T})}} \,dt \right|&,
\end{align*}
and the conclusion follows by the convergence ${u}^n \rightharpoonup^* {u}$ in $L^\infty(0,T; \tilde{H}^s(\Omega_T))$ (using $(- \Delta)^s\varphi \in L^1(0,T;H^{-s}(\Omega_T))$ as test function for the convergence in \eqref{eq: convergence u^n to u}).

\smallskip
Finally, for the third limit \eqref{subeq: lim equation 3}, it suffices to consider
\begin{align*}
    &\left| \int_{0}^{T}\langle \nabla W(u^n(t)), \varphi(t) \rangle_{L^2(\Omega_{T})} \,dt - \int_{0}^{T} \scal{\nabla W(u(t))}{\varphi(t)}_{L^2(\Omega_{T})} \,dt 
    \right|
    \\
    &\qquad\qquad\leq  \left|\int_{0}^{T} \scal{\nabla W(u^n(t)) - \nabla W(u(t)) }{\varphi(t)}_{L^2(\Omega_{T})}  \,dt \right|
    \\
    &\qquad\qquad \leq \int_{0}^{T} \norm{\nabla W(u^n(t)) - \nabla W(u(t))}_{L^2(\Omega_{T})} \norm{\varphi(t)}_{L^2(\Omega_{T})}  \,dt
    \\
    &\qquad\qquad \leq \sqrt{T} \norm{\varphi}_{L^\infty(0,T;L^2(\Omega_{T}))} \left( \int_{0}^{T} \norm{\nabla W(u^n(t)) - \nabla W(u(t))}^2_{L^2(\Omega_{T})} \,dt \right)^\frac{1}{2}
    \\
    &\qquad\qquad \leq  C_1 \sqrt{T} \norm{\varphi}_{L^\infty(0,T;L^2(\Omega_{T}))} \norm{u^n-u}_{C^0([0,T];L^2(\Omega_T))}, 
\end{align*}
where we have used Cauchy Schwarz both in $L^2(\Omega_T)$ and on $(0,T)$ and the positive constant $C_1$ is given by
\begin{align*}
    \left(\int_{0}^{T} \norm{\nabla W(u^n(t)) - \nabla W(u(t))}^2_{L^2(\Omega_{T})} \right)^\frac{1}{2} \,dt &\overset{\eqref{eq: nabla W}}{\leq} \left(\int_{0}^{T} \left( \int_{\Omega_T} \mathrm{K}^2 |u^n(t) - u(t)|^2 \,dx \right) \,dt  \right)^\frac{1}{2}
    \\
    & \,\, \leq \underbrace{\mathrm{K} \sqrt{T}}_{C_1} \norm{u^n-u}_{C^0([0,T];L^2(\Omega_T))}.
\end{align*}
Hence, the convergence of $u^n$ to $u$ provided by \eqref{eq: convergence u^n to u in C0} yields the validity of \eqref{subeq: lim equation 3}. Thus, \eqref{subeq: lim equation} holds.

\smallskip
To conclude the proof of \eqref{def eq u weak sol}, we just note that the space
\[
\{\varphi \in C^{\infty}_{c}((0,T)\times \Omega_T) \colon \, \supp \varphi \subset \mathcal{O} \}
\]
is dense in the space
\[
\{ \varphi \in L^1(0,T;\tilde{H}^s(\Omega_T)) \cap W^{1,1}(0,T; L^2(\Omega_T)) \colon \, \supp \varphi \subset \mathcal{O} \}
\]
(a density result in the same spirit can be found in \cite[Lemma 2.8]{dal2019cauchy}). Therefore, the validity of (iii) in Definition \ref{defweaksolu} has been established.

Finally, we are left to prove the energy inequality \eqref{eq: en ineq for u}. We first note that by inequality \eqref{ineq for W(x)-W(y)} in Remark \ref{property of W}, one can easily deduce that for a.e. $t \in [0,T]$ 
\begin{equation}\label{eq: W u^n - W u}
    \norm{W(u^n(t)) - W(u(t))}_{L^1(\Omega_t)}
    \leq \tilde{C} \norm{u^n(t) - u(t)}_{L^2(\Omega_T)}
\end{equation} 
for a suitable constant $\tilde{C} >0$ independent of $t$ and $n$, due to the boundedness of $u^n$ and $u$ in $L^\infty(0,T;L^2(\Omega_T))$ (see \cite[Prop. 8 (iii)]{bonafini2021obstacle}).
We now integrate \eqref{ineq: en bal u^n 2} between arbitrary times $\alpha,\beta$ with $0 \leq \alpha \leq \beta \leq T$. By \eqref{eq: convergence u^n to u} and \eqref{eq: convergence u^n to u in C0} and by standard lower semicontinuity
arguments along with estimate \eqref{eq: W u^n - W u}, as $n \to + \infty$ we get
\begin{align*}
&\int_{\alpha}^{\beta} \frac12 \norm{\dot{u}(t)}^2_{L^2(\Omega_{t})}+\frac12 [u(t)]_{s}^2+\|W(u(t))\|_{L^1(\Omega_{t})} \, dt
\\
&\leq (\beta-\alpha) \left( \frac12 \norm{v_0}^2_{L^2(\Omega_{0})} +\frac12 [u_0]_{s}^2+\|W(u_0)\|_{L^1(\Omega_{0})} \right) + \int_{\alpha}^{\beta} \int_{0}^{t} \scal{f(\tau)}{\dot{u}^n(\tau)}_{L^2(\Omega_{\tau})} \,d\tau \, dt.
\end{align*}
By the arbitrariness of $\alpha$ and $\beta$, the above inequality yields the validity of \eqref{eq: en ineq for u} for a.e. $t \in [0,T]$.

In order to provide the energy inequality for every time, we note that the energy estimates can be replicated in order to obtain the following: $\dot{u} \in \cscal([\overline{t},T]; L^2(\Omega_{\overline{t}}))$ for every $\overline{t} \in [0,T)$ (cf. the existence result in the Appendix \ref{appendix}). Hence, we fix a time $\overline{t} \in [0,T)$ and consider a decreasing sequence $t_k \searrow \overline{t}$ along which the energy inequality is satisfied (which is possible since we already proved it holds a.e. in $[0,T]$). Then, again by weak lower semicontinuity arguments along with estimate \eqref{eq: W u^n - W u}, we deduce that
\begin{equation*}
\begin{split}
    &\frac12 \|\dot{u}(\overline{t})\|^{2}_{L^{2}(\Omega_{\overline{t}})}+\frac12 [u(\overline{t})]_{s}^2+\|W(u(\overline{t}))\|_{L^{1}(\Omega_{\overline{t}})} 
    \\
    &\leq \liminf_{k \to \infty} \left( \frac12 \|\dot{u}(t_k)\|^{2}_{L^{2}(\Omega_{t_k})}+\frac12 [u(t_k)]_{s}^2+\|W(u(t_k))\|_{L^{1}(\Omega_{t_k})}  \right) 
    \\
    &\leq
    E_0 + \int_{0}^{\overline{t}}\scal{f(\tau)}{\dot{u}(\tau)}_{L^2(\Omega_\tau)}\, d\tau.
\end{split}
\end{equation*}
Hence, \eqref{eq: en ineq for u} holds for every $\overline{t} \in [0,T)$. Finally, the case $\overline{t}= T$ can be treated just by considering a final time $\tilde{T}>T$, setting $\Omega_t:= \Omega_T$ for every $t \in [T,\tilde{T}]$ and arguing as above. This concludes the proof.
\end{proof}

\section{Penalty method}\label{sec: penalty}
In this section, we provide another proof of Theorem \ref{thm main result} by means of a penalty argument, whose idea goes back to \cite{Lions64}. The strategy consists of first considering a parameter $\varepsilon > 0$ and constructing an approximate solution $u_{\varepsilon}$ (obtained by the standard Galerkin method) to problem \eqref{eq:u}, modified by adding a penalty term of the form $\varepsilon^{-1} \mathscr{P} \dot{u}$. Here, $\mathscr{P}$ is the characteristic function of the complement of the set $\mathcal{O}$. The family $\{u_\varepsilon\}_\varepsilon$ admits a cluster point $u$, which solves \eqref{eq:u}; the penalization term forces $u$ to vanish outside $\mathcal{O}$.
We mentioned that, in the same spirit, other types of penalties are possible: see for example \cite{zolesio1990approximation}, where the (linear) wave equation for the Laplace operator in non-cylindrical domains is treated.

\begin{proof}[Proof of Theorem \ref{thm main result}]
$\,$

\medskip
\noindent
\emph{Step 1. Penalized problems at level $\varepsilon$.}
For every fixed $\varepsilon > 0$, we consider the following formal penalized problem 
\begin{equation}\label{eq:peneq}
\begin{cases}
\ddot{u} + (-\Delta)^s u + \nabla W(u) + \frac{1}{\varepsilon} \mathscr{P} \dot{u} = f & \text{in } (0,T) \times \Omega_T,
\\
u = 0  &\text{in } (0,T) \times (\R^d \setminus\Omega_T),
\end{cases}
\end{equation}
with initial conditions given by $u(0)=u_0$ and $\dot{u}(0)=v_0$, and $\mathscr{P}$ given by
\begin{equation*}
\mathscr{P} (t,x)=\begin{cases}
0 &\text{ if } (t,x) \, \in \, \mathcal{O},
\\
1 & \, \text{otherwise}.
\end{cases}
\end{equation*}
We aim to prove the existence of a weak solution $u_{\varepsilon} \colon [0,T] \times \R^d \to \R^m$ to problem \eqref{eq:peneq}, in the sense that
\begin{equation}\label{eq: weak u eps penalized}
\begin{aligned}
-\int_{0}^{T} \langle \dot{u}_{\varepsilon}(t), \dot{\varphi}(t) \rangle_{L^2(\Omega_T)} \, dt
+\int_{0}^{T}[u_{\varepsilon}(t), \varphi(t)]_s \,dt 
+\int_{0}^{T} \scal{\nabla W(u_{\varepsilon}(t))}{\varphi(t)}_{L^2(\Omega_{T})} \,dt 
\\ + \frac{1}{\varepsilon} \int_{0}^{T} \scal{\mathscr{P}(t) \, \dot{u}_{\varepsilon}(t)}{\varphi(t)}_{L^2(\Omega_{T})} \,dt
= \int_{0}^{T} \langle f(t), \varphi(t) \rangle_{L^2(\Omega_T)}\, dt,
\end{aligned}
\end{equation}
for every $\varphi \in \mathcal{X}$, where
\[
\mathcal{X} := \{ \varphi \in L^1(0,T;\tilde{H}^s(\Omega_T)) \cap W^{1,1}(0,T; L^2(\Omega_T)) \colon \, \varphi(0)=\varphi(T)=0 \}.
\]
Moreover, $u_{\varepsilon}$ is also required to satisfy the energy inequality
\begin{equation}\label{energy epsilon level}
\begin{aligned}
\frac{1}{2}\norm{\dot{u}_{\varepsilon}(t)}^2_{L^2(\Omega_{T})} + \frac{1}{2}[u_{\varepsilon}(t)]_{s}^2 +\norm{W(u_\varepsilon)}_{L^1(\Omega_{T})} 
+\frac{1}{\varepsilon}\int_0^t\int_{\Omega_T} \mathscr{P}(\tau,x) |\dot{u}_{\varepsilon}(\tau,x)|^2 \,dxd\tau
\\
\leq E_0 + \int_{0}^{t}\scal{f(\tau)}{\dot{u}_\varepsilon(\tau)}_{L^2(\Omega_T)}\, d\tau,
\end{aligned}
\end{equation}
where $E_0$ is given by Definition \ref{def energy}. Existence of such a solution $u_\varepsilon$ is obtained via a Galerkin approximation, hence we need to introduce a sequence of approximating spaces.

\medskip
\noindent
\emph{Step 2. Finite dimensional approximating spaces.} Let $\lbrace w_{k} \rbrace_{k\in \N}$ be an orthogonal basis of $\tilde{H}^{s}(\Omega_T)$, and at the same time an orthonormal basis of $L^2(\Omega_T)$. More specifically, we choose the eigenfunctions of the fractional Laplacian on $\Omega_T$, namely
\begin{equation*}
    \begin{cases}
        (-\Delta)^s w_k = \lambda_k w_k &\text{on } \Omega_T,
        \\
        w_k \in \tilde{H}^s(\Omega_T),
    \end{cases}
\end{equation*}
which are orthogonal in $\tilde{H}^s(\Omega_T)$ and orthonormal in $L^2(\Omega_T)$
Note that, if $\varphi \in \tilde{H}^s(\Omega_T)$, then for all $k \in \N$ it holds
\begin{equation*}
    [\varphi,w_k]_s = \int_{\R^{d}} (-\Delta)^{s/2}\varphi(x) (-\Delta)^{s/2}w_k(x) \, dx = \int_{\R^{d}} \varphi(x) (-\Delta)^{s}w_k(x) \, dx = \lambda_k \langle \varphi, w_k \rangle_{L^2(\Omega_T)}.
\end{equation*}
Hence, for all $k \in \N$ we have
\begin{equation}\label{eq: frequency phi w_j}
    \langle \varphi, w_k \rangle_{L^2(\Omega_T)} = \lambda^{-1}_k [\varphi,w_k]_s \qquad \text{for every } \varphi \in \tilde{H}^s(\Omega_T).
\end{equation}
In particular, \eqref{eq: frequency phi w_j} yields $[w_k]^2_s = [w_k,w_k]_s = \lambda_k$ for every $k \in \N$.

\smallskip
For each $M > 0$, define $V_M \subset \tilde{H}^{s}(\Omega_T)$ to be the subspace 
\[
V_M := \textup{span} \,\{w_1, \dots, w_M\}.
\]
Let $\pi_M$ be the $L^2(\Omega_T)$-orthogonal projection onto the subspace $V_M$, namely
\begin{equation*}
	\pi_M: \tilde{H}^s(\Omega_T) \to V_M, \, \phi \mapsto \pi_M(\phi):= \sum_{k = 1}^{M} \langle \phi, w_k \rangle_{L^2(\Omega_T)} w_k.    
\end{equation*}
Note that for any $\phi \in \tilde{H}^s(\Omega_T)$, it holds that
\begin{equation*}
	\begin{split}
		\norm{\phi - \pi_M(\phi)}^2_{L^2(\Omega_T)} = \norm{ \sum_{k = M+1}^{\infty} \langle \phi, w_k \rangle_{L^2(\Omega_T)} w_k }^2_{L^2(\Omega_T)} \leq \sum_{k = M+1}^{\infty} |\langle \phi, w_k \rangle_{L^2(\Omega_T)}|^2 \norm{w_k}^2_{L^2(\Omega_T)}
		\\
		= \sum_{k = M+1}^{\infty} \frac{1}{\lambda_k} \lambda_k |\langle \phi, w_k \rangle_{L^2(\Omega_T)}|^2 \leq \frac{1}{\lambda_M} 
		\sum_{k = 1}^{\infty} \lambda_k |\langle \phi, w_k \rangle_{L^2(\Omega_T)}|^2 =  \frac{1}{\lambda_M} [\phi]_s^2 \, \leq  \frac{1}{\lambda_M} \norm{\phi}_s^2 \, ,
	\end{split}
\end{equation*} 
so that we have the decay estimate
\begin{equation}\label{eq: estimate phi - pi_n phi}
	\norm{\phi - \pi_M(\phi)}_{L^2(\Omega_T)} \leq \frac{1}{\sqrt{\lambda_M}} \norm{\phi}_s \qquad \text{for every } \phi \in \tilde{H}^s(\Omega_T).
\end{equation}

\medskip
\noindent
\emph{Step 3. Galerkin-penalized approximants $u_{M,\varepsilon}$.} 
Now, let $\{a^k\}_{k \in \N}$ and $\{b^k\}_{k\in \N}$ be the coefficients of $u_0$ and $v_0$ with respect to the given basis $\lbrace w_{k} \rbrace_{k \in \N}$. For every $M > 0$, define
\begin{equation*}
\begin{aligned}
	 u_{0,M} := \sum_{k=1}^{M} a^k w_k
	 \quad\text{and}\quad
	 v_{0,M} := \sum_{k=1}^{M} b^k w_k.
\end{aligned}
\end{equation*}
We seek for a $V_M$-valued weak solution $u_{M,\varepsilon}$ to \eqref{eq:peneq}, namely we seek for a coefficients function $d_{M,\varepsilon} \colon [0,T] \to \mathbb{R}^M, t \mapsto (d^1_{M,\varepsilon}(t),\dots,d^M_{M,\varepsilon}(t))$ such that the function given by 
\begin{equation}\label{eq: def u M eps}
    u_{M,\varepsilon}(t,x)= \sum_{k=1}^M \, d^k_{M,\varepsilon}(t)\, w_k(x) \qquad \text{for }t \in [0,T], x \in \Omega_T,
\end{equation}
is a solution of  
\begin{equation}\label{eq:weakeq}
\begin{split}
    \left\langle \ddot{u}_{M,\varepsilon}(t), \phi \right\rangle_{L^2(\Omega_{T})}   +[u_{M,\varepsilon}(t), \phi]_{s} & + \left\langle \nabla W(u_{M,\varepsilon}(t)), \phi \right\rangle_{L^2(\Omega_{T})}
    \\
    & + \frac{1}{\varepsilon} \left\langle \mathscr{P} (t) \, \dot{u}_{M,\varepsilon}(t), \phi \right\rangle_{L^2(\Omega_{T})} 
    = \langle f(t), \phi \rangle_{L^2(\Omega_{T})} 
\end{split}
\end{equation}
for a.e. $t \in [0,T]$ and for every $\phi \in V_M$, coupled with initial conditions given by 
\begin{subequations}
\begin{align}
    &u_{M,\varepsilon}(0) = u_{0,M}, \label{eq:u0M} \\
    &\dot u_{M,\varepsilon}(0) = v_{0,M}. \label{eq:v0M}
\end{align}
\end{subequations}

By testing \eqref{eq:weakeq} with $\phi = w_1, \dots, w_M$, the components $(d^1_{M,\varepsilon},\ldots,d_{M,\varepsilon}^M)$ of $d_{M,\varepsilon}$ satisfy the following ODE system
\begin{equation}\label{ODEs}
\begin{cases}
    \displaystyle
    \ddot{d}^k_{M, \varepsilon}(t) + \mathcal{A}^k(d_{M,\varepsilon}(t)) + \mathcal{W}^k(d_{M,\varepsilon}(t)) + \mathcal{R}^k(t,\dot{d}_{M,\varepsilon}(t)) =\mathcal{F}^k(t) \quad k = 1, \dots, M,
    \\
    d_{M,\varepsilon}(0)=(a^1, \dots, a^M),
    \\
    \dot{d}_{M,\varepsilon}(0)=(b^1, \dots, b^M),
\end{cases}
\end{equation}
where, for every $y =(y^1,\ldots,y^M)\in \mathbb{R}^M$, $z =(z^1,\ldots,z^M)\in \mathbb{R}^M$ and $k=1,\dots,M$, we set
\begin{equation*}
\begin{aligned}
    & \mathcal{A}^k(y) := [y^k w_{k},w_{k}]_{s} = y^k [w_{k}]^2_{s} = \lambda_k y^k
    \\
    &\mathcal{W}^k(y) := \int_{\Omega_T} \nabla W \left(\sum_{j=1}^M \, y^j\, w_j(x) \right) \cdot w_k(x) \,dx,
    \\
    &\mathcal{R}^k(t,z) := \frac{1}{\varepsilon} \int_{\Omega_T} \mathscr{P}(t,x) \left(\sum_{j=1}^M z^j \, w_j(x) \right) \cdot w_k(x) \,dx, 
    \\
    &\mathcal{F}^k(t) := \int_{\Omega_T} f(t,x) \cdot w_k(x) \,dx.
\end{aligned}
\end{equation*} 
We now prove that \eqref{ODEs} admits a classical solution $d_{M,\varepsilon} \in W^{2,\infty}(0,T; \R^M)$. Let
\begin{equation*}
g \colon [0,T] \times \mathbb{R}^{2M} \to \mathbb{R}^{2M}, \quad (t,(y,z)) \to (z, (g^1(t,y,z), \dots, g^M(t,y,z))),
\end{equation*}
with $g^k$ be given by 
\begin{equation*}
    g^k(t,y,z) := \mathcal{F}^k(t) - \mathcal{A}^k(y) - \mathcal{W}^k(y) - \mathcal{R}^k(t,z)  \quad k=1,\dots,M.
\end{equation*}
We have that:
\begin{itemize}
	\item[(i)] $g\colon [0,T] \times \mathbb{R}^{2M} \to \mathbb{R}^{2M}$ is a Carathéodory function;
    
	\item[(ii)] there exists a constant $C>0$ (possibly depending on $M$ and $\varepsilon$) such that
	\begin{equation*}
		|g(t,y,z)| \leq C (1 + |y| + |z|) \quad \text{for every } (t,(y,z)) \in [0,T] \times R^{2M}.
	\end{equation*}
\end{itemize}
Indeed, claim (i) follows by the observation that for almost every $t \in [0,T]$ the map $(y,z) \mapsto g(t,(y,z))$ is continuous, thanks to the assumption on the continuity of $\nabla W$ and the linearity with respect to $y$ and $z$ of $\mathcal{A}^k$ and $\mathcal{R}^k$, respectively. Secondly, for every fixed $(y,z) \in \R^{2M}$, the map $t \mapsto g(t,(y,z))$ is measurable: the measurability of the function $t \mapsto \mathcal{R}^k(t,z)$ and the measurability of the function $t \mapsto \mathcal{F}^k(t)$ follows from Fubini's Theorem.

On the other hand, claim (ii) holds due to the following estimates:
\begin{equation*}
\begin{aligned}
    &|\mathcal{F}^k(t)| \leq \norm{f(t)}_{L^2(\Omega_T))} \norm{w_k}_{L^2(\Omega_T))} \leq \|f\|_{L^\infty(0,T;L^2(\Omega_T))} \quad \text{for any } t \in [0,T],
    \\
    &|\mathcal{R}^k(t,z)| \leq \frac{1}{\varepsilon} \sum_{j=1}^M  |z^j| \int_{\Omega_T} |w_j(x)| |w_k(x)| \, dx \leq \frac{1}{\varepsilon} \sum_{j=1}^M  |z^j| \quad \text{for any } t \in [0,T], z \in \R^M,
\end{aligned}
\end{equation*}
and 	
\begin{equation*}
    \begin{aligned}
    |\mathcal{W}^k(y)| &\leq \left| \int_{\Omega_T} \nabla W \left(\sum_{j=1}^M \, y^j\, w_j(x) \right) \cdot w_k(x) \,dx \right|
    \leq \int_{\Omega_T} C_W \left(1+ \sum_{j=1}^M \left| y^j\, w_j(x) \right| \right) |w_k(x)| \, dx
    \\
    & \leq C_W \sqrt{\mathcal{L}^d(\Omega_T)}  + C_W  \sum_{j=1}^M  \left( |y^j| \int_{\Omega_T} |w_j(x)| |w_k(x)| \, dx \right)
    \\
    & \leq \max \left\{ C_W \sqrt{\mathcal{L}^d(\Omega_T)}, C_W \right\} \left(1+ \sum_{j=1}^M  |y^j|\right) \quad \text{for any } y \in \R^M,
    \end{aligned}
\end{equation*}
where we have used the Cauchy Schwarz inequality in $L^2(\Omega_T)$, the orthonormality of the system $\lbrace w_{k} \rbrace_{k\in \N}$ in $L^2(\Omega_T)$ and inequality \eqref{ineq for nabla W(y)} in Remark \ref{property of W}. Hence, claim (ii) is proven, recalling the equivalence of norms on the finite-dimensional spaces.

Then, we apply the existence result provided by \cite[Thm. A.2]{SolombrinoFornasier} (see also \cite[Chap. 1, Thms. 1 \& 2]{Filippov1960differential}), and we obtain the existence of a classical solution $d_{M,\varepsilon} \in W^{2,\infty}(0,T; \R^M)$ to \eqref{ODEs}. In particular, the function $u_{M,\varepsilon}$ given by \eqref{eq: def u M eps} satisfies \eqref{eq:weakeq} with initial conditions \eqref{eq:u0M} and \eqref{eq:v0M}. Moreover,
\begin{equation*}
    u_{M,\varepsilon} \in C^0([0,T]; V_M), \, \dot{u}_{M,\varepsilon} \in C^0([0,T]; V_M) \text{ and } \ddot{u}_{M,\varepsilon} \in L^\infty(0,T; V_M).
\end{equation*}

\medskip
\noindent
\emph{Step 4. Energy estimates.} We first observe that, by $d_{M,\varepsilon} \in W^{2,\infty}(0,T; \R^M)$ and by considering $\phi = \dot{u}_{M,\varepsilon}$ in equation \eqref{eq:weakeq} (see also \cite[Remark 2]{bonafini2021obstacle}), a direct computation yields
\begin{equation}\label{approidentity}
\begin{aligned}
\frac{d}{dt} \left( \,\frac12 \norm{\dot{u}_{M,\varepsilon}(t)}^2_{L^2(\Omega_{T})}+\frac12 [u_{M,\varepsilon}(t)]_{s}^2
+\|W(u_{M,\varepsilon}(t))\|_{L^1(\Omega_{T})} \right) &\\
+\frac{1}{\varepsilon}\int_{\Omega_T} \mathscr{P}(t,x) |\dot{u}_{M,\varepsilon}(t,x)|^2dx & =\langle f(t), \dot{u}_{M,\varepsilon}(t) \rangle_{L^2(\Omega_{T})}.
\end{aligned}
\end{equation}
Then, integrating \eqref{approidentity} between $0$ and $t$, $0\leq t \leq T$, we obtain
\begin{equation}\label{energyestimates1}
\begin{aligned}
\frac12 \norm{\dot{u}_{M,\varepsilon}(t)}^2_{L^2(\Omega_{T})}+\frac12 [u_{M,\varepsilon}(t)]_{s}^2
+\|W(u_{M,\varepsilon}(t))\|_{L^1(\Omega_{T})}
+\frac{1}{\varepsilon}\int_0^t\int_{\Omega_T} \mathscr{P}(\tau,x) |\dot{u}_{M,\varepsilon}(\tau,x)|^2 \,dxd\tau \\
\leq \frac12 \norm{v_{0,M}}^2_{L^2(\Omega_{T})} + \frac12 [u_{0,M}]_{s}^2 +\|W(u_{0,M})\|_{L^1(\Omega_{T})}+ \int_0^t \langle f(t), \dot{u}_{M,\varepsilon}(t) \rangle_{L^2(\Omega_{T})}.
\end{aligned}
\end{equation}
We note that
\[
\lim_{M \to +\infty} \frac12 \norm{v_{0,M}}^2_{L^2(\Omega_{T})} + \frac12 [u_{0,M}]_{s}^2 +\|W(u_{0,M})\|_{L^1(\Omega_{T})} = \underbrace{\frac12 \norm{v_0}^2_{L^2(\Omega_{0})} +\frac12 [u_0]_{s}^2+\|W(u_0)\|_{L^1(\Omega_{T})}}_{E_0}.
\]
Hence, there exists $M^* > 0$ large enough, independent of $\varepsilon$, such that 
\begin{equation*}
\frac12 \norm{v_{0,M}}^2_{L^2(\Omega_{T})} + \frac12 [u_{0,M}]_{s}^2 +\|W(u_{0,M})\|_{L^1(\Omega_{T})}\leq 2 E_0  \quad \text{for every } M \geq M^*. 
\end{equation*}
In particular, for every $M \geq M^*$, we obtain that 
\begin{equation}\label{energyestimates2}
\begin{aligned}
\,\frac12 \norm{\dot{u}_{M,\varepsilon}(t)}^2_{L^2(\Omega_{T})}+\frac12 [u_{M,\varepsilon}(t)]_{s}^2
+\|W(u_{M,\varepsilon}(t))\|_{L^1(\Omega_{T})}
+\frac{1}{\varepsilon}\int_0^t\int_{\Omega_T} \mathscr{P}(\tau,x) |\dot{u}_{M,\varepsilon}(\tau,x)|^2 \,dxd\tau \\
\leq 2 E_0+ \int_0^t \langle f(\tau), \dot{u}_{M,\varepsilon}(\tau) \rangle_{L^2(\Omega_{T})}\,d\tau.
\end{aligned}
\end{equation}
By means of the same Gronwall's argument used to derive \eqref{ineq: gron argument} from \eqref{eq:step0gronwall}, we can prove that there exists a constant $C > 0$ independent of $\varepsilon$ and $M$, such that, for all $M > M^*$,
\begin{equation}\label{ineq: gron argument II}
	\sup_{t \in [0,T]} \left( \norm{\dot{u}_{M,\varepsilon}(t)}^2_{L^2(\Omega_{T})} + [u_{M,\varepsilon}(t)]_{s}^2  \right) \leq C.
\end{equation}
In turn, there exists a positive constant $c>0$ independent of $M, \varepsilon$ and $t$ such that, for all $M > M^*$,
\begin{equation}\label{energyestimates3}
	\frac12 \norm{\dot{u}_{M,\varepsilon}(t)}^2_{L^2(\Omega_{T})}+\frac12 [u_{M,\varepsilon}(t)]_{s}^2
	+\|W(u_{M,\varepsilon}(t))\|_{L^1(\Omega_{T})}
	+\frac{1}{\varepsilon}\int_0^t\int_{\Omega_T} \mathscr{P}(\tau,x)|\dot{u}_{M,\varepsilon}(\tau,x)|^2 \, dx d\tau \leq c.
\end{equation}

\medskip
\noindent
\emph{Step 5. Limits as $M \to +\infty$.} For every fixed $\varepsilon > 0$, the uniform bound in \eqref{ineq: gron argument II} yields the existence of a function $u_{\varepsilon} \in L^\infty(0,T; \tilde{H}^s(\Omega_T))$ with $\dot{u}_{\varepsilon} \in L^\infty(0,T; L^2(\Omega_T))$  such that
\begin{equation}\label{compactness1}
\begin{aligned}
& u_{M,\varepsilon} \rightharpoonup^* u_\varepsilon \quad \text{ in } L^\infty(0,T; \tilde{H}^s(\Omega_T)),
\\
& \dot{u}_{M,\varepsilon} \rightharpoonup^* \dot{u}_{\varepsilon} \quad \text{ in } L^\infty(0,T; L^2(\Omega_T)).
\end{aligned}
\end{equation}
By integrating \eqref{eq:weakeq} in time over $[0,T]$, we can deduce that, for each $M > M^*$, we have
\begin{equation}\label{eq: weak un penalized}
\begin{aligned}
-\int_{0}^{T} \langle \dot{u}_{M,\varepsilon}(t), \dot{\varphi}_M(t) \rangle_{L^2(\Omega_T)} \, dt
+\int_{0}^{T}[u_{M,\varepsilon}(t), \varphi_M(t)]_s \,dt 
+\int_{0}^{T} \scal{\nabla W(u_{M,\varepsilon}(t))}{\varphi_M(t)}_{L^2(\Omega_{T})} \,dt 
\\ +  \frac{1}{\varepsilon}  \int_{0}^{T}\scal{\mathscr{P}(t) \, \dot{u}_{M,\varepsilon}(t)}{\varphi_M(t)}_{L^2(\Omega_{T})} \,dt
= \int_{0}^{T} \langle f(t), \varphi_M(t) \rangle_{L^2(\Omega_T)}\, dt,
\end{aligned}
\end{equation}
for every $\varphi_M \in \mathcal{Y}_M$, where $\mathcal{Y}_M := C^{\infty}_c(0,T; V_M)$. Moreover, we note that $\{\mathcal{Y}_M\}_M$ defines an increasing sequence of subspaces of $\mathcal{Y} := C^{\infty}_c(0,T; \tilde{H}^s(\Omega_T))$, with union dense in $\mathcal{Y}$.
Both $\mathcal{Y}_M$ and $\mathcal{Y}$ are clearly subspaces of $\mathcal{X}$.

In order to prove \eqref{eq: weak u eps penalized}, we now pass to the limit in \eqref{eq: weak un penalized}. To do so, fix a function $\varphi \in \mathcal{Y}$ with $\supp \varphi \subset \mathcal{O}$. Let $\{\varphi_M\}_{M \in \N} \subset \mathcal{Y}_M$ be the sequence of functions given by
\begin{equation}\label{eq: def of varphi_n}
\varphi_M (t) := \pi_M(\varphi(t)) =  \sum_{k = 1}^{M} \langle \varphi(t), w_k \rangle_{L^2(\Omega_T)} w_k \in V_M \quad \text{for a.e. }t \in [0,T],
\end{equation}
for which we also have that
\begin{equation*}
\dot{\varphi}_M (t) = \pi_M(\dot{\varphi}(t)) =  \sum_{k = 1}^{M} \langle \dot{\varphi}(t), w_j \rangle_{L^2(\Omega_T)} w_j \in V_M \quad \text{for a.e. }t \in [0,T].
\end{equation*}
One can prove that, by construction and by Dominated Convergence Theorem, the sequence of functions $\{\varphi_M\}_{M \in \N}$ converges in $\mathcal{X}$ to $\varphi$. 
Then, we prove the following:
\begin{subequations}\label{subeq: lim}
	\begin{align}
	\label{subeq: lim 1} \lim_{M \to +\infty} \int_{0}^{T} \langle \dot{u}_{M,\varepsilon}(t), \dot{\varphi}_M(t) \rangle_{L^2(\Omega_T)} \, dt &= \int_{0}^{T} \langle \dot{u}_{\varepsilon}(t), \dot{\varphi}(t) \rangle_{L^2(\Omega_T)} \, dt,
	\\
	\label{subeq: lim 2} \lim_{M \to +\infty}\int_{0}^{T}[u_{M,\varepsilon}(t), \varphi_M(t)]_s \,dt &= \int_{0}^{T}[u_{\varepsilon}(t), \varphi(t)]_s \,dt,
	\\
	\label{subeq: lim 3} \lim_{M \to +\infty}\int_{0}^{T} \scal{\nabla W(u_{M,\varepsilon}(t))}{\varphi_M(t)}_{L^2(\Omega_{T})} \,dt &= \int_{0}^{T} \scal{\nabla W(u_{\varepsilon}(t))}{\varphi(t)}_{L^2(\Omega_{T})} \,dt,  
	\\ 
    \label{subeq: lim 4}\lim_{M \to +\infty}\int_{0}^{T} \scal{\mathscr{P}(t)\dot{u}_{M,\varepsilon}(t)}{\varphi_M(t)}_{L^2(\Omega_{T})} \,dt &= \int_{0}^{T} \scal{\mathscr{P}(t)\dot{u}_{\varepsilon}(t)}{\varphi(t)}_{L^2(\Omega_{T})} \,dt,
	\\
	  \label{subeq: lim 5} \lim_{M \to +\infty}\int_{0}^{T} \langle f(t), \varphi_M(t) \rangle_{L^2(\Omega_T)}\, dt &= \int_{0}^{T} \langle f(t), \varphi(t) \rangle_{L^2(\Omega_T)}\, dt. 
	\end{align}
\end{subequations}

\smallskip
For the first limit \eqref{subeq: lim 1}, it suffices to consider the estimate
\begin{equation*}
\begin{split}
&\left| \int_{0}^{T} \langle \dot{u}_{M,\varepsilon}(t), \dot{\varphi}_M(t) \rangle_{L^2(\Omega_T)} \, dt 
- 
\int_{0}^{T} \langle \dot{u}_{\varepsilon}(t), \dot{\varphi}(t) \rangle_{L^2(\Omega_T)} \, dt \right|
\\
&\leq \left| \int_{0}^{T} \langle \dot{u}_{M,\varepsilon}(t) -\dot{u}_{\varepsilon}(t), \dot{\varphi}(t) \rangle_{L^2(\Omega_T)} \, dt \right| 
+ \int_{0}^{T} \norm{\dot{u}_{M,\varepsilon}(t)}_{L^2(\Omega_T)} \norm{\dot{\varphi}_M(t) - \dot{\varphi}(t)}_{L^2(\Omega_T)} \, dt,
\\
&\leq \left| \int_{0}^{T} \langle \dot{u}_{M,\varepsilon}(t) -\dot{u}_{\varepsilon}(t), \dot{\varphi}(t) \rangle_{L^2(\Omega_T)} \, dt \right| 
+ \sqrt{C} \norm{\dot{\varphi}_M(t) - \dot{\varphi}(t) }_{L^{1}(0,T;L^2(\Omega_T))} ,
\end{split}
\end{equation*}
and the convergence $\dot{u}_{M,\varepsilon} \rightharpoonup^* \dot{u}_{\varepsilon}$ in $L^\infty(0,T; L^2(\Omega_T))$ (using $\dot{\varphi} \in L^1(0,T;L^2(\Omega_T))$ as test function for the convergence in \eqref{compactness1}) and the convergence of $\varphi_M$ to $\varphi$ in $\mathcal{X}$ along with the uniform bound for $\dot{u}_{M,\varepsilon}$ provided by the constant $C>0$ in \eqref{ineq: gron argument II}.

\smallskip
For the second limit \eqref{subeq: lim 2}, we note that $\varphi \in \mathcal{Y}$ implies $(-\Delta)^s \varphi \in L^1(0,T; H^{-s}(\Omega_T))$, hence by triangle inequality and Cauchy Schwarz inequality for the seminorm $[\cdot,\cdot]_s$ and the inequality $[\cdot]_s \leq \|\cdot\|_s$ in $\tilde{H}^s(\Omega_{T})$, we get
\begin{equation*}
\begin{split}
&\left|\int_{0}^{T}[u_{M,\varepsilon}(t), \varphi_M(t)]_s \,dt - \int_{0}^{T}[u_{\varepsilon}(t), \varphi(t)]_s \,dt
\right|
\\
&\leq \left| \int_{0}^{T}[u_{M,\varepsilon}(t) - u_{\varepsilon}(t), \varphi(t)]_s \,dt  \right| + \left| \int_{0}^{T}[u_{M,\varepsilon}(t), \varphi_M(t)-\varphi(t)]_s \,dt \right|
\\
&\leq \left| \int_{0}^{T} \scal{(- \Delta)^s\varphi(t)}{u_{M,\varepsilon}(t) - u_{\varepsilon}(t)}_{{H^{-s}(\Omega_{T}), \tilde{H}^s(\Omega_{T})}} \,dt \right|  
+  \int_{0}^{T} [u_{M,\varepsilon}(t)]_s \, [\varphi_M(t)-\varphi(t)]_s \,dt 
\\
&\leq \left| \int_{0}^{T} \scal{(- \Delta)^s\varphi(t)}{u_{M,\varepsilon}(t) - u_{\varepsilon}(t)}_{{H^{-s}(\Omega_{T}), \tilde{H}^s(\Omega_{T})}} \,dt \right| 
+ \sqrt{C} \norm{\varphi_M - \varphi}_{L^1(0,T;\tilde{H}^s(\Omega_T))}  ,
\end{split}
\end{equation*}
and the conclusion follows by the convergence ${u}_{M,\varepsilon} \rightharpoonup^* {u}_{\varepsilon}$ in $L^\infty(0,T; \tilde{H}^s(\Omega_T))$ (using $(- \Delta)^s\varphi \in L^1(0,T;H^{-s}(\Omega_T))$ as test function for the convergence in \eqref{compactness1} and the convergence of $\varphi_M$ to $\varphi$ in $\mathcal{X}$ along with the equi-boundedness of $u_{M,\varepsilon}$ provided by the constant $C>0$ in \eqref{ineq: gron argument II}.

\smallskip
For the third limit \eqref{subeq: lim 3}, by the equi-boundedness of $\{\dot{u}_{M,\varepsilon}\}_M$ in $L^\infty(0,T; L^{2}(\Omega_T))$ together with an Ascoli--Arzelà argument for $\{u_{M,\varepsilon}\}_M$ in $C^0([0,T]; L^2(\Omega_T))$, we deduce that
\begin{equation}\label{eq: convergence u^n to u in C0 II}
u_{M,\varepsilon} \to u_{\varepsilon} \text{ in } C^0([0,T]; L^2(\Omega_T)).
\end{equation} 
Then, we can proceed in a similar way and we can apply triangle inequality and Cauchy Schwarz inequality in $L^2(\Omega_T)$, obtaining 
\begin{align*}
\left| \int_{0}^{T} \right.&\left.\scal{\nabla W(u_{M,\varepsilon}(t))}{\varphi_M(t)}_{L^2(\Omega_{T})} \,dt - \int_{0}^{T} \scal{\nabla W(u_{\varepsilon}(t))}{\varphi(t)}_{L^2(\Omega_{T})} \,dt 
\right|
\\
&\leq  \left|\int_{0}^{T} \scal{\nabla W(u_{M,\varepsilon}(t)) - \nabla W(u_{\varepsilon}(t)) }{\varphi(t)}_{L^2(\Omega_{T})}  \,dt \right|
\\
& \quad + \left| \int_{0}^{T}  \scal{\nabla W(u_{M,\varepsilon}(t))}{\varphi_M(t)-\varphi(t)}_{L^2(\Omega_{T})}  \,dt \right|
\\
& \leq \int_{0}^{T} \norm{\nabla W(u_{M,\varepsilon}(t)) - \nabla W(u_{\varepsilon}(t))}_{L^2(\Omega_{T})} \norm{\varphi(t)}_{L^2(\Omega_{T})}  \,dt
\\
& \quad + \int_{0}^{T}  \norm{\nabla W(u_{M,\varepsilon}(t))}_{L^2(\Omega_{T})} \norm{\varphi_M(t)-\varphi(t)}_{L^2(\Omega_{T})}  \,dt 
\\
& \leq \sqrt{T}\norm{\varphi}_{L^\infty(0,T;L^2(\Omega_{T}))} \left( \int_{0}^{T} \norm{\nabla W(u_{M,\varepsilon}(t)) - \nabla W(u_{\varepsilon}(t))}^2_{L^2(\Omega_{T})} \,dt \right)^\frac{1}{2}
\\
& \quad + \norm{\varphi_M - \varphi}_{L^1(0,T;L^2(\Omega_{T}))}  \norm{\nabla W(u_{M,\varepsilon})}_{L^\infty(0,T;L^2(\Omega_{T}))}  
\\
& \leq C_1 \sqrt{T} \norm{\varphi}_{L^\infty(0,T;L^2(\Omega_{T}))} \norm{u_{M,\varepsilon}-u_{\varepsilon}}_{C^0([0,T];L^2(\Omega_T))} + C_2 \norm{\varphi_M - \varphi}_{L^1(0,T;L^2(\Omega_{T}))}, 
\end{align*}
where the two positive constants $C_1$ and $C_2$ coming from the terms containing $\nabla W$ in the last inequality are given by the following straightforward estimates (using the properties of $W$ and the Cauchy Schwarz inequality):
\begin{align*}
    \left(\int_{0}^{T} \norm{\nabla W(u_{M,\varepsilon}(t)) - \nabla W(u_{\varepsilon}(t))}^2_{L^2(\Omega_{T})} \,dt \right)^{\frac{1}{2}} &\overset{\eqref{eq: nabla W}}{\leq} \left(\int_{0}^{T} \left( \int_{\Omega_T} \mathrm{K}^2 |u_{M,\varepsilon}(t) - u_{\varepsilon}(t)|^2 \,dx \right) \,dt\right)^{\frac{1}{2}}
    \\
    & \,\, \leq \underbrace{\mathrm{K} \sqrt{T \mathcal{L}^d(\Omega_T)}}_{C_1} \norm{u_{M,\varepsilon} - u_{\varepsilon}}_{C^0([0,T];L^2(\Omega_T))},
\end{align*}
and
\begin{align*}
    \esssup_{t \in (0,T)} \norm{\nabla W(u_{M,\varepsilon}(t))}_{L^2(\Omega_{T})}  \,dt & \overset{\eqref{ineq for nabla W(y)}}{\leq}  \esssup_{t \in (0,T)} \left( C_W \int_{\Omega_T} (1+|u_{M,\varepsilon}(t,x)|)^2 \,dx \right)^\frac{1}{2}
    \\
    \text{\footnotesize (using $(a+b)^2 \leq 2a^2+2b^2$)} \quad
    & \leq \underbrace{  \sqrt{2C_W} \left( \mathcal{L}^d(\Omega_T) + \norm{u_{M,\varepsilon}}^2_{L^\infty(0,T;L^2(\Omega_T))} \right)^\frac{1}{2}}_{C_2}
\end{align*}
Hence, summarizing the above estimates allows us to conclude that \eqref{subeq: lim 3} holds, thanks to the convergence provided by \eqref{eq: convergence u^n to u in C0 II} and the convergence of $\varphi_n$ to $\varphi$ in $\mathcal{X}$.

\smallskip
For the fourth limit \eqref{subeq: lim 4}, it suffices to estimate 
\begin{equation*}
\begin{split}
&\left| \int_{0}^{T} \langle \mathscr{P}(t)
\dot{u}_{M,\varepsilon}(t), \varphi_M(t) \rangle_{L^2(\Omega_T)} \, dt 
- 
\int_{0}^{T} \langle \mathscr{P}(t) \dot{u}_{\varepsilon}(t), \varphi(t) \rangle_{L^2(\Omega_T)} \, dt \right|
\\
&\leq \left| \int_{0}^{T} \langle \dot{u}_{M,\varepsilon}(t) -\dot{u}_{\varepsilon}(t), \mathscr{P}(t) \varphi(t) \rangle_{L^2(\Omega_T)} \, dt \right| 
+ \int_{0}^{T} \norm{\mathscr{P}(t) \dot{u}_{M,\varepsilon}(t)}_{L^2(\Omega_T)} \norm{\varphi_M(t) - \varphi(t)}_{L^2(\Omega_T)} \, dt,
\\
&\leq \left| \int_{0}^{T} \langle \dot{u}_{M,\varepsilon}(t) -\dot{u}_{\varepsilon}(t), \mathscr{P}(t) \varphi(t) \rangle_{L^2(\Omega_T)} \, dt \right| 
+ \sqrt{C} \norm{\varphi_M - \varphi }_{L^{1}(0,T;L^2(\Omega_T))} ,
\end{split}
\end{equation*}
and conclude by the convergence $\dot{u}_{M,\varepsilon} \rightharpoonup^* \dot{u}_{\varepsilon}$ in $L^\infty(0,T; L^2(\Omega_T))$ (using $\mathscr{P} \varphi \in L^1(0,T;L^2(\Omega_T))$ as test function for the convergence in \eqref{compactness1}) and the convergence of $\varphi_M$ to $\varphi$ in $\mathcal{X}$ along with the uniform bound for $\dot{u}_{M,\varepsilon}$ provided by the constant $C>0$ in \eqref{ineq: gron argument II}.

\smallskip
Finally, for the fifth limit \eqref{subeq: lim 5}, simply consider
\begin{equation*}
\begin{split}
\left| \int_{0}^{T} \langle f(t), \varphi_M(t) \rangle_{L^2(\Omega_T)}\, dt \right. & - \left. \int_{0}^{T} \langle f(t), \varphi(t) \rangle_{L^2(\Omega_T)}\, dt \right| 
\\
&\leq \norm{f}_{L^\infty(0,T;L^2(\Omega_T))} \, \norm{\varphi_M - \varphi}_{L^1(0,T;L^2(\Omega_{T}))},
\end{split}
\end{equation*}
and again the conclusion follows by the convergence of $\varphi_M$ to $\varphi$ in $\mathcal{X}$ and the hypothesis on $f$ (cf. \eqref{eq:f}). 
Therefore, passing to the limit in \eqref{eq: weak un penalized} and using the limits provided by \eqref{subeq: lim}, we can conclude that $u_{\varepsilon}$ satisfies \eqref{eq: weak u eps penalized}, for every fixed $\varepsilon>0$.

\medskip
\noindent
\emph{Step 6. Initial values for $u_\varepsilon$.}
In order to prove that $\dot u_\varepsilon(0) = v_0$, an additional uniform bound on second derivatives is needed, in our case a bound of the sequence $\{\ddot{u}_{M,\varepsilon}\}_{M} \subset L^\infty(0,T; H^{-s}(\Omega_0))$, for every fixed $\varepsilon >0$. Since the space $\mathcal{X}_0$, given by 
\[
\mathcal{X}_0 := \{ \varphi \in L^1(0,T;\tilde{H}^s(\Omega_0)) \cap W^{1,1}(0,T; L^2(\Omega_0)) \colon \, \varphi(0)=\varphi(T)=0 \} \subset \mathcal{X},
\]
is dense in $L^1(0,T; \tilde{H}^{s}(\Omega_0))$, it suffices to provide an estimate for the action of every $\ddot{u}_{M,\varepsilon}$ against every $\varphi \in \mathcal{X}_0$. Since $\ddot{u}_{M,\varepsilon}(t) \in L^2(\Omega_0)$ for almost every $t \in [0,T]$, one has  
\begin{equation}\label{eq: ddot u penalty}
		\int_{0}^{T}\scal{\ddot{u}_{M,\varepsilon}(t)}{\varphi(t)}_{H^{-s}(\Omega_0),\tilde{H}^s(\Omega_{0})}\,dt 
        = -\int_{0}^{T}\scal{\dot{u}_{M,\varepsilon}(t)}{\dot{\varphi}(t)}_{L^2(\Omega_{0})} \, dt.
\end{equation}
We now rewrite the right hand side of the previous equation: using that $\varphi(t) = 0$ in $\Omega_T \setminus \Omega_0$, adding and subtracting the term $\varphi_M$ (cf. \eqref{eq: def of varphi_n}), using the orthogonality to $V_M$ of the function $\dot{\varphi}_M(t) - \dot{\varphi}(t)$, and finally using \eqref{eq: weak un penalized} for the admissible choice of test functions $\varphi_M$, we obtain
\begin{align*}
    &\int_{0}^{T}\scal{\ddot{u}_{M,\varepsilon}(t)}{\varphi(t)}_{H^{-s}(\Omega_0),\tilde{H}^s(\Omega_{0})}\,dt
    = -\int_{0}^{T}\scal{\dot{u}_{M,\varepsilon}(t)}{\dot{\varphi}(t)}_{L^2(\Omega_{T})} \, dt
    \\
    &
    = -\int_{0}^{T}\scal{\dot{u}_{M,\varepsilon}(t)}{\dot{\varphi}_M(t)}_{L^2(\Omega_{T})} \, dt - \underbrace{\int_{0}^{T}\scal{\dot{u}_{M,\varepsilon}(t)}{\dot{\varphi}(t) - \dot{\varphi}_M(t)}_{L^2(\Omega_{T})} \, dt}_{=0}
    \\
    & 
    =
    \int_{0}^{T} \langle f(t), \varphi_M(t) \rangle_{L^2(\Omega_T)}\, dt 
    - \int_{0}^{T}[u_{M,\varepsilon}(t), \varphi_M(t)]_s \,dt 
    - \int_{0}^{T} \scal{\nabla W(u_{M,\varepsilon}(t))}{\varphi_M(t)}_{L^2(\Omega_{T})} \,dt 
    \\
    &
    \quad - \int_{0}^{T} \frac{1}{\varepsilon} \scal{\mathscr{P}(t) \, \dot{u}_{M,\varepsilon}(t)}{\varphi_M(t)}_{L^2(\Omega_{T})} \,dt. 
\end{align*}
Now, reasoning in a similar way as before and essentially replacing $\varphi_M-\varphi$ simply with $\varphi$ in the estimates above for proving \eqref{subeq: lim}, one can prove that inequality \eqref{ineq for nabla W(y)} in Remark \ref{property of W}, the regularity of $f$ (see \eqref{eq:f}) and the equi-boundedness of $u_{M,\varepsilon}$ in $L^\infty(0,T; \tilde{H}^s(\Omega_T))$ and in $W^{1,\infty}(0,T; L^2(\Omega_T))$ (cf. \eqref{ineq: gron argument}) yield the existence of positive constants $D_1,D_2$ and $D_3$ such that the following inequalities hold:
\begin{equation}\label{ineq estimate for ddot u M eps D1D2D3}
    \begin{aligned}
    \left| \int_{0}^{T}[u_{M,\varepsilon}(t), \varphi_M(t)]_s \,dt  \right|
    & \leq D_1 \norm{\varphi}_{L^1(0,T; \tilde{H}^s(\Omega_0))},
    \\
    \left| \int_{0}^{T} \scal{\nabla W(u_{M,\varepsilon}(t))}{\varphi_M(t)}_{L^2(\Omega_{T})} \,dt \right| 
    & \leq D_2 \norm{\varphi}_{L^1(0,T; L^2(\Omega_0))},
    \\
    \left| \int_{0}^{T} \langle f(t), \varphi_M(t) \rangle_{L^2(\Omega_T)}\, dt  \right| 
    & \leq D_3  \norm{\varphi}_{L^1(0,T; L^2(\Omega_0))}.
\end{aligned}
\end{equation}
Note that we used the standard property $\norm{\varphi_M}_{L^1(0,T; \tilde{H}^s(\Omega_0))} \leq \norm{\varphi}_{L^1(0,T; \tilde{H}^s(\Omega_0))}$.

Furthermore, we can prove that there exists a positive constant $D_4$, such that 
\begin{equation}\label{ineq estimate for ddot u M eps D4}
    \left|  \frac{1}{\varepsilon}  \int_{0}^{T}\scal{\mathscr{P}(t) \, \dot{u}_{M,\varepsilon}(t)}{\varphi_M(t)}_{L^2(\Omega_{T})} \,dt \right| \leq D_4  \norm{\varphi}_{L^1(0,T; \tilde{H}^s(\Omega_0))}.
\end{equation}
To this aim, it suffices to consider the following inequalities:
\begin{align*}
	\left| \frac{1}{\varepsilon} \int_{0}^{T}  \scal{\mathscr{P}(t) \, \dot{u}_{M,\varepsilon}(t)}{\varphi_M(t)}_{L^2(\Omega_{T})} \,dt \right|
	&\leq \underbrace{ \left|  \frac{1}{\varepsilon} \int_{0}^{T}  \scal{\mathscr{P}(t) \, \dot{u}_{M,\varepsilon}(t)}{\varphi(t)}_{L^2(\Omega_{T})} \,dt  \right|}_{=0}
	\\
	& \, + \left|  \frac{1}{\varepsilon} \int_{0}^{T} \scal{\mathscr{P}(t) \, \dot{u}_{M,\varepsilon}(t)}{\varphi_M(t) - \varphi(t)}_{L^2(\Omega_{T})} \,dt  \right|
	\\
	\text{\footnotesize (by Cauchy Schwarz in $L^2(\Omega_T)$ and $|\mathscr{P}(t,x)| \leq 1$)} \ & \leq \frac{1}{\varepsilon}  \int_{0}^{T} \norm{\dot{u}_{M,\varepsilon}(t)}_{L^2(\Omega_T)} \norm{\varphi_M(t) - \varphi(t)}_{L^2(\Omega_{T})} \,dt
	\\
	\text{\footnotesize (by Cauchy Schwarz in $[0,T]$ and by \eqref{eq: estimate phi - pi_n phi})} \, 
	& \leq \frac{1}{\varepsilon} \norm{\dot{u}_{M,\varepsilon}}_{L^\infty(0,T;L^2(\Omega_T))} \int_{0}^{T} \frac{1}{\sqrt{\lambda_M}} \norm{\varphi(t)}_{\tilde{H}^s(\Omega_T)} \,dt 
	\\
	\text{\footnotesize (by \eqref{ineq: gron argument II})} \, & \leq \frac{1}{\varepsilon \sqrt{\lambda_M}} \sqrt{C} \norm{\varphi}_{L^1(0,T;\tilde{H}^s(\Omega_T))}.
\end{align*}
Note that, for every fixed $\varepsilon > 0$, the last term goes to zero as $M$ goes to $+\infty$. In particular, this implies the validity of \eqref{ineq estimate for ddot u M eps D4}. Hence, by \eqref{ineq estimate for ddot u M eps D1D2D3} and \eqref{ineq estimate for ddot u M eps D4}, for every fixed $\varepsilon > 0$ we have that
\[
\left| \int_{0}^{T}\scal{\ddot{u}_{M,\varepsilon}(t)}{\varphi(t)}_{H^{-s}(\Omega_0),\tilde{H}^s(\Omega_{0})}\,dt  \right| \leq D \norm{\varphi}_{L^1(0,T; \tilde{H}^s(\Omega_0))}
\]
for a suitable constant $D>0$, thus $\{\ddot{u}_{M,\varepsilon}\}_M$ is equi-bounded in $L^\infty(0,T; H^{-s}(\Omega_0))$.

Eventually, equi-boundedness of $\{\ddot{u}_{M,\varepsilon}\}_M$ in $L^\infty(0,T; H^{-s}(\Omega_0))$ allows to apply an Ascoli--Arzelà argument for $\{\dot u_{M,\varepsilon}\}_M \subset C^0([0,T]; H^{-s}(\Omega_0))$, and we conclude that
\begin{equation}\label{eq: convergence u^n to u in C0 penalty}
    \dot{u}_{M,\varepsilon} \to \dot{u}_{\varepsilon} \text{ in } C^0([0,T]; H^{-s}(\Omega_0)).
\end{equation}
Furthermore, by \cite[Lemma 8.1]{lionsmagenes} we have that
\begin{equation*}
    L^\infty(0,T;\tilde{H}^s(\Omega_T)) \cap C^0([0,T]; L^2(\Omega_T)) \subseteq \cscal([0,T]; \tilde{H}^s(\Omega_T)), 
\end{equation*}
hence we also obtain $u_{\varepsilon} \in \cscal([0,T]; \tilde{H}^s(\Omega_T))$. Similarly, we have
\begin{equation*}
    L^\infty(0,T;L^2(\Omega_0)) \cap C^0([0,T]; H^{-s}(\Omega_0)) \subseteq \cscal([0,T]; L^2(\Omega_0)), 
\end{equation*}
hence $\dot{u}_{\varepsilon} \in \cscal([0,T]; L^2(\Omega_{0}))$. 
In particular, using the initial data for $u_{M,\varepsilon}$ (see \eqref{eq:u0M}) and by \eqref{eq: convergence u^n to u in C0 II} and \eqref{eq: convergence u^n to u in C0 penalty}, one obtains that for every fixed $\varepsilon>0$ the function $u_\varepsilon$ obtained in \eqref{compactness1} satisfies
\begin{equation*}
    u_{\varepsilon}(0) = u_0  \text{ and } \dot{u}_{\varepsilon}(0) = v_0.
\end{equation*}

We are left to prove the validity of the energy inequality \eqref{energy epsilon level}. 
This follows by the very same argument by lower semicontinuity at the end of the proof in Section \ref{sec: time discr}, replacing \eqref{ineq: en bal u^n 2} with \eqref{energyestimates1}.

\medskip
\noindent
\emph{Step 7. Limit as $\varepsilon \to 0$.} We only sketch the outline, since the general argument follows the same steps we implemented for passing to the limit in $M$. First, we observe that estimates \eqref{energyestimates3} and \eqref{ineq: gron argument II} extend uniformly in $\varepsilon$. In particular, we have pre-compactness of the sequence $\{u_\varepsilon\}_{\varepsilon}$ in  $L^\infty(0,T; \tilde{H}^s(\Omega_T))$ and $W^{1,\infty}(0,T; L^2(\Omega_T))$. 
Secondly,  one can prove that $\{\ddot{u}_\varepsilon\}_\varepsilon$ is equi-bounded in $L^\infty(0,T; H^{-s}(\Omega_0))$ uniformly with respect to $\varepsilon$, using \eqref{eq: weak u eps penalized} and the same estimates as in \eqref{eq: ddot u penalty} and \eqref{ineq estimate for ddot u M eps D1D2D3}. Notice that
\begin{equation*}
    \frac{1}{\varepsilon} \int_{0}^{T} \scal{\mathscr{P}(t) \, \dot{u}_{\varepsilon}(t)}{\varphi(t)}_{L^2(\Omega_{T})} \,dt = 0 \quad \text{for all } \varepsilon >0 \text{ and } \varphi \in L^1(0,T; \tilde{H}^s(\Omega_0)).
\end{equation*}
Hence, possibly passing to a subsequence, we deduce the existence of a function 
\[
u \in L^\infty(0,T; \tilde{H}^s(\Omega_T)) \cap W^{1,\infty}(0,T; L^2(\Omega_T))
\]
such that
\begin{equation}\label{compactness for u}
\begin{aligned}
& u_{\varepsilon} \rightharpoonup^* u \quad \text{ in } L^\infty(0,T; \tilde{H}^s(\Omega_T)), && u_{\varepsilon} \to u \quad \text{ in } C^0([0,T]; L^2(\Omega_T)),
\\
& \dot{u}_{\varepsilon} \rightharpoonup^* \dot{u} \quad \text{ in } L^\infty(0,T; L^2(\Omega_T)), && \dot{u}_{\varepsilon} \to \dot{u} \quad \text{ in } C^0([0,T]; H^{-s}(\Omega_0)).
\end{aligned}
\end{equation}
In particular, we obtain from \eqref{energy epsilon level} that there exists a positive constant $C$ such that
\begin{equation*}
    \frac{1}{\varepsilon}\int_0^t\int_{\Omega_T} \mathscr{P}(\tau,x) |\dot{u}_{\varepsilon}(\tau,x)|^2 \,dxd\tau \leq C.
\end{equation*}
This implies that $\dot{u}=0$ a.e. on $((0,T) \times \Omega_T) \setminus \mathcal{O}$. Since $u(0,x)=0$ for a.e. $x \in \R^d \setminus \Omega_0$, $u \in C^0([0,T]; L^2(\Omega_T))$, and by assumption $\Omega_t$ is increasing, i.e., $\R^d \setminus \Omega_t$ is decreasing, we conclude that $u=0$ a.e. on $((0,T) \times \Omega_T) \setminus \mathcal{O}$, hence $u$ satisfies (i) in Definition \ref{defweaksolu}. 

Finally, using estimates similar to those used to prove \eqref{subeq: lim}, one can pass to the limit in \eqref{eq: weak u eps penalized} using \eqref{compactness for u}: we deduce that $u$ satisfies (iii) in Definition \ref{defweaksolu}, for every function $\varphi \in \mathcal{X}$ with $\supp \varphi \subset \mathcal{O}$. Note that
\begin{equation*}
    \frac{1}{\varepsilon} \int_{0}^{T} \scal{\mathscr{P}(t) \, \dot{u}_{\varepsilon}(t)}{\varphi(t)}_{L^2(\Omega_{T})} \,dt = 0 \quad \text{for all } \varepsilon >0 \text{ and } \varphi\in \mathcal{X}, \supp \varphi \subset \mathcal{O}.
\end{equation*}

Finally, one has $u \in \cscal([0,T]; \tilde{H}^s(\Omega_T))$ and $\dot{u} \in \cscal([0,T]; L^2(\Omega_0))$, hence the initial data for $u_\varepsilon$ pass to the limit and $u$ satisfies (ii) of Definition \ref{defweaksolu}. Moreover, the energy inequality for $u$ follows from the same lower semicontinuity argument at the end of the proof in Section \ref{sec: time discr}, this time replacing \eqref{ineq: en bal u^n 2} with \eqref{energy epsilon level}.
This concludes the proof.
\end{proof}

\appendix
\section{Fractional semilinear wave equations in cylindrical domains}\label{appendix}
In this appendix we recall a result on the existence of weak solutions to the fractional semilinear wave equation, which we generalize by adding a non-zero forcing term $f$. The proof can be done by a constructive time-discrete variational scheme whose main ideas date back to \cite{Ro30} (see also \cite{ambrosio1995minimizing} for the minimizing movements approach introduced by De Giorgi) and which have, since then, been adapted to many instances of parabolic and hyperbolic equations. For instance, the homogeneous fractional wave equations (i.e., with no potential term and no forcing term) have been studied in \cite{bonafini2019variational} (see also \cite{bonafini2022weak,bonafini2021obstacle}). The general case with a forcing term can be tackled by defining a suitable approximant for the function $f$ (see \cite[Thm. 3.1]{dal2019cauchy}).

Let $T, s > 0$ and $m, d \in \N$. Let $\Omega \subset \R^d$ be an open bounded set with Lipschitz boundary. For a function $u \colon [0,T] \times \R^d \to \R^m$, let us consider the problem
\begin{equation}\label{eq:semiwaveswithforce}
\begin{cases}
\ddot{u} + (-\Delta)^s u +\nabla W(u)= f    
& \text{in } (0,T) \times \Omega,                
\\
u = 0                               	
& \text{in } [0,T] \times (\R^d \setminus \Omega),    \\
u(0,x) = u_0(x)                          	
& \text{for }  x \in \Omega,                         
\\
\dot{u}(0,x) = v_0(x) 
&\text{for }  x \in \Omega.
\end{cases}
\end{equation}
where: 
\begin{subequations}
	\begin{align}
		\raisebox{0.25ex}{\tiny$\bullet$} \, \, & W \in C^1(\R^m) \, \text{is a non-negative potential with Lipschitz continuous gradient, } \nonumber 
		\\
		&\text{i.e., $\exists \mathrm{K} > 0$ such that } |\nabla W(x) - \nabla W(y)| \leq \mathrm{K} |x-y| \text{ for all } x,y \in \R^m,
		\label{app eq: nabla W}
		\\
		\raisebox{0.25ex}{\tiny$\bullet$} \, \, & f\in L^\infty(0,T; L^2(\R^d)) \text{ with } \supp f \subset (0,T) \times \Omega, \label{app eq: f}
		\\
		\raisebox{0.25ex}{\tiny$\bullet$} \, \, &u_0\in \tilde{H}^s(\Omega) \text{ and }
		v_0\in L^2(\Omega). \label{app eq: u_0 u_1}
	\end{align}
\end{subequations}

We define a weak solution to \eqref{eq:semiwaveswithforce} as follows.
\begin{definition}\label{def:weak}
	We say that a function $u \colon (0, T)\times \R^{d}\to \R^{m}$ is a weak solution to problem \eqref{eq:semiwaveswithforce} if the following holds:
	\begin{enumerate}
		\item[(i)]
		$u \in L^\infty(0,T; \tilde{H}^s(\Omega))$, $\dot{u} \in L^{\infty}(0,T;L^2(\Omega))$
        and $\ddot{u} \in L^{\infty}(0,T;H^{-s}(\Omega))$;
	
		\item[(ii)] for every $\varphi \in L^1(0,T; \tilde{H}^s(\Omega)) \cap W^{1,1}(0,T;L^2(\Omega))$ one has
		\begin{equation*}
		\begin{split}
		-\int_{0}^T \scal{\dot{u}(t)}{\dot{\varphi}(t)}_{L^2(\Omega)} dt + \int_{0}^T [ u(t), \varphi(t) ]_{s} \, dt +\int_{0}^T \scal{ \nabla W(u(t))}{\varphi(t)}_{L^2(\Omega)} \,dt 
		\\
		= \scal{\dot{u}(0)}{\varphi(0)}_{L^2(\Omega)} -\scal{\dot{u}(T)}{\varphi(T)}_{L^2(\Omega)} + \int_{0}^T \scal{f(t)}{\varphi(t)}_{L^2(\Omega)}\,dt
		\end{split}
		\end{equation*}
		with
		\begin{equation*}
		u(0)=u_{0} \quad \text{ and } \quad \dot{u}(0)=v_{0}.
		\end{equation*}
	\end{enumerate}
\end{definition}

Finally, we introduce the energy definition for a weak solution to problem \eqref{eq:semiwaveswithforce}. 
\begin{definition}
	Let $u$ be a weak solution to \eqref{eq:semiwaveswithforce}. We define the energy of $u$ as
	\begin{equation*}
	E(u(t)) = \frac12 \|\dot{u}(t)\|^{2}_{L^{2}(\Omega)}+\frac12 [u(t)]_{s}^2+\|W(u(t))\|_{L^{1}(\Omega)}, \quad \text{for every }t \in [0,T].
	\end{equation*}
\end{definition}

The following theorem can be proved following the same strategy used in \cite[Thm. 3]{bonafini2021obstacle}, with the only difference consisting in the treatment of the forcing term $f$ and taking into account \cite[Lemma 8.1]{lionsmagenes}.
\begin{theorem}\label{app:thmexistence}
	There exists a weak solution $u$ to problem \eqref{eq:semiwaveswithforce} in the sense of Definition \ref{def:weak}. Moreover, the following energy inequality holds:
	\begin{equation*}
	E(u(t)) \leq E(u(0)) + \int_{0}^{t} \scal{f(\tau)}{\dot{u}(\tau)}_{L^2(\Omega)} \,d\tau \qquad \text{for every } t \in [0,T].
	\end{equation*}
	Furthermore,
	\begin{equation*}
		u \in C^0([0,T];L^2(\Omega)) \cap \cscal([0,T];\tilde{H}^s(\Omega))  \quad \text{and} \quad \dot{u} \in C^0([0,T];H^{-s}(\Omega))\cap \cscal([0,T];L^2(\Omega)).
    \end{equation*}
\end{theorem}

\subsection*{Acknowledgment}
M.\,Bonafini and R.\,Molinarolo are members of the ``Gruppo Nazionale per l'Analisi Matematica, la Probabilit\`a e le loro Applicazioni'' (GNAMPA) of the ``Istituto Nazionale di Alta Matematica'' (INdAM). 

V.\,P.\,C.\,Le gratefully acknowledges the support of the Department of Mathematics at Heidelberg University through a postdoctoral position, as well as the support of STRUCTURES, Cluster of Excellence at Heidelberg University.

R.\,Molinarolo is supported by the Project ``Giochi a campo medio, trasporto e ottimizzazione in sistemi auto-organizzati e machine learning”, funded by MUR, D.D. 47/2025, PNRR - Missione 4, Componente 2, Investimento 1.2 - funded by European Union NextGenerationEU, CUP B33C25000380001. R.\,Molinarolo acknowledges the support of the project funded by the EuropeanUnion-NextGenerationEU under the National
Recovery and Resilience Plan (NRRP), Mission 4 Component 2 Investment 1.1 - Call PRIN 2022 No. 104 of February 2, 2022 of Italian Ministry of University and Research; Project 2022SENJZ3 ``Perturbation problems and asymptotics for elliptic differential equations: variational and potential theoretic methods." 

The authors wish to thank Prof. G.\,Orlandi for suggesting this problem and for his interest in this work. We also acknowledge Profs. G.\,Canevari and F.\,Solombrino for fruitful discussions and valuable advices. 

\bibliographystyle{plain}
\bibliography{bibliography}

@article {DiNezzaPalatucciValdinoci12,
    AUTHOR = {Di Nezza, Eleonora and Palatucci, Giampiero and Valdinoci,
              Enrico},
     TITLE = {Hitchhiker's guide to the fractional {S}obolev spaces},
   JOURNAL = {Bull. Sci. Math.},
  FJOURNAL = {Bulletin des Sciences Math\'ematiques},
    VOLUME = {136},
      YEAR = {2012},
    NUMBER = {5},
     PAGES = {521--573},
      ISSN = {0007-4497},
   MRCLASS = {46E35 (35A23 35S05 35S30)},
  MRNUMBER = {2944369},
MRREVIEWER = {Lanzhe Liu},
       DOI = {10.1016/j.bulsci.2011.12.004},
       URL = {http://dx.doi.org/10.1016/j.bulsci.2011.12.004},
}

@article {Schatzman78,
    AUTHOR = {Schatzman, Michelle},
     TITLE = {A class of nonlinear differential equations of second order in
              time},
   JOURNAL = {Nonlinear Anal.},
  FJOURNAL = {Nonlinear Analysis},
    VOLUME = {2},
      YEAR = {1978},
    NUMBER = {3},
     PAGES = {355--373},
      ISSN = {0362-546X},
   MRCLASS = {34A60},
  MRNUMBER = {512664},
MRREVIEWER = {Yu. V. Trubnikov},
       DOI = {10.1016/0362-546X(78)90022-6},
       URL = {http://dx.doi.org/10.1016/0362-546X(78)90022-6},
}

@article {Maruo85,
    AUTHOR = {Maruo, Kenji},
     TITLE = {Existence of solutions of some nonlinear wave equations},
   JOURNAL = {Osaka Journal of Mathematics},
  FJOURNAL = {Osaka Journal of Mathematics},
    VOLUME = {22},
      YEAR = {1985},
    NUMBER = {1},
     PAGES = {21--30},
      ISSN = {0030-6126},
   MRCLASS = {35L70 (34G20)},
  MRNUMBER = {785518},
MRREVIEWER = {P. W. Bates},
       URL = {http://projecteuclid.org/euclid.ojm/1200778031},
}

@article{ambrosio1995minimizing,
  title={Minimizing movements},
  author={Ambrosio, Luigi},
  journal={Rend. Accad. Naz. Sci. XL Mem. Mat. Appl.(5)},
  volume={19},
  pages={191--246},
  year={1995}
}

@book{mclean2000strongly,
  title={Strongly elliptic systems and boundary integral equations},
  author={McLean, William and McLean, William Charles Hector},
  year={2000},
  publisher={Cambridge university press}
}

@article{Ro30,
  title={Zweidimensionale parabolische randwertaufgaben als grenzfall eindimensionaler randwertaufgaben},
  author={Rothe, Erich},
  journal={Mathematische Annalen},
  volume={102},
  number={1},
  pages={650--670},
  year={1930},
  publisher={Springer}
}

@article{SeTi,
  title={Nonlinear wave equations as limits of convex minimization problems: proof of a conjecture by {D}e {G}iorgi},
  author={Serra, Enrico and Tilli, Paolo},
  journal={Annals of Mathematics},
  pages={1551--1574},
  year={2012},
  publisher={JSTOR}
}

@article{jerrard2011defects,
  title={Defects in semilinear wave equations and timelike minimal surfaces in {M}inkowski space},
  author={Leon Jerrard, Robert},
  journal={Analysis \& PDE},
  volume={4},
  number={2},
  pages={285--340},
  year={2011},
  publisher={Mathematical Sciences Publishers}
}

@article{bellettini2010time,
  title={Time-like minimal submanifolds as singular limits of nonlinear wave equations},
  author={Bellettini, Giovanni and Novaga, Matteo and Orlandi, Giandomenico},
  journal={Physica D: Nonlinear Phenomena},
  volume={239},
  number={6},
  pages={335--339},
  year={2010},
  publisher={Elsevier}
}

@article{bonafini2019variational,
  title={A variational scheme for hyperbolic obstacle problems},
  author={Bonafini, Mauro and Novaga, Matteo and Orlandi, Giandomenico},
  journal={Nonlinear Analysis},
  volume={188},
  pages={389--404},
  year={2019},
  publisher={Elsevier}
}

@article{Neu,
  title={Kinks and the minimal surface equation in {M}inkowski space},
  author={John Charles Neu},
  journal={Physica D: Nonlinear Phenomena},
  volume={43},
  pages={421-434},
  year={1990},
}

@article{SmailyJerrad,
  title={A refined description of evolving interfaces in certain nonlinear wave equations},
  author={Mohammad Ibrahim El Smaily and Robert Leon Jerrard},
  journal={NoDEA Nonlinear Differential Equations Appl},
  volume={25},
  pages={no.2, Art. 15, 21 pp},
  year={2018},
}

@article{L.C.Evans,
  title={Partial {D}ifferential {E}quations },
  author={Lawrence Craig Evans},
  journal={American Mathematical Society, Providence, RI},
  year={2010},
}

@article{dal2019cauchy,
  title={On the Cauchy problem for the wave equation on time-dependent domains},
  author={Dal Maso, Gianni and Toader, Rodica},
  journal={Journal of Differential Equations},
  volume={266},
  number={6},
  pages={3209--3246},
  year={2019},
  publisher={Elsevier}
}

@article{bonafini2021obstacle,
  title={On the obstacle problem for fractional semilinear wave equations},
  author={Bonafini, Mauro and Le, Van Phu Cuong and Novaga, Matteo and Orlandi, Giandomenico},
  journal={Nonlinear Analysis},
  volume={210},
  pages={112368},
  year={2021},
  publisher={Elsevier}
}

@article{bonafini2022weak,
  title={Weak solutions for nonlinear waves in adhesive phenomena},
  author={Bonafini, Mauro and Le, Van Phu Cuong},
  journal={Annali dell'Università di Ferrara},
  volume={68},
  number={1},
  pages={223--233},
  year={2022},
  publisher={Springer}
}

@article {CalvoNovagaOrlandi,
    AUTHOR = {Calvo, Juan and Novaga, Matteo and Orlandi, Giandomenico},
     TITLE = {Parabolic equations in time-dependent domains},
   JOURNAL = {J. Evol. Equ.},
  FJOURNAL = {Journal of Evolution Equations},
    VOLUME = {17},
      YEAR = {2017},
    NUMBER = {2},
     PAGES = {781--804},
      ISSN = {1424-3199,1424-3202},
   MRCLASS = {35K59 (35K20 35K65 47H05)},
  MRNUMBER = {3665229},
       DOI = {10.1007/s00028-016-0336-4},
       URL = {https://doi.org/10.1007/s00028-016-0336-4},
}

@article{LazzaroniMolinaroloRivaSolombrino,
    AUTHOR = {Lazzaroni, Giuliano and Molinarolo, Riccardo and Riva, Filippo
              and Solombrino, Francesco},
     TITLE = {On the wave equation on moving domains: regularity, energy
              balance and application to dynamic debonding},
   JOURNAL = {Interfaces Free Bound.},
  FJOURNAL = {Interfaces and Free Boundaries. Mathematical Analysis,
              Computation and Applications},
    VOLUME = {25},
      YEAR = {2023},
    NUMBER = {3},
     PAGES = {401--454},
      ISSN = {1463-9963,1463-9971},
   MRCLASS = {35R37 (35L05 35L85)},
  MRNUMBER = {4642019},
       DOI = {10.4171/ifb/485},
       URL = {https://doi.org/10.4171/ifb/485},
}

@article{SolombrinoFornasier,
     author = {Fornasier, Massimo and Solombrino, Francesco},
     title = {Mean-Field {Optimal} {Control}},
     journal = {ESAIM: Control, Optimisation and Calculus of Variations},
     pages = {1123--1152},
     publisher = {EDP-Sciences},
     volume = {20},
     number = {4},
     year = {2014},
     doi = {10.1051/cocv/2014009},
     mrnumber = {3264236},
     language = {en},
     url = {http://www.numdam.org/articles/10.1051/cocv/2014009/}
}

@book {Filippov1960differential,
    AUTHOR = {Filippov, A. F.},
     TITLE = {Differential equations with discontinuous righthand sides},
    SERIES = {Mathematics and its Applications (Soviet Series)},
    VOLUME = {18},
 PUBLISHER = {Kluwer Academic Publishers Group, Dordrecht},
      YEAR = {1988},
     PAGES = {x+304},
      ISBN = {90-277-2699-X},
   MRCLASS = {34-02 (58F99)},
  MRNUMBER = {1028776},
       DOI = {10.1007/978-94-015-7793-9},
       URL = {https://doi.org/10.1007/978-94-015-7793-9},
}

@article{LAZZARONI2022112822,
title = {Radial solutions for a dynamic debonding model in dimension two},
journal = {Nonlinear Analysis},
volume = {219},
pages = {112822},
year = {2022},
issn = {0362-546X},
doi = {https://doi.org/10.1016/j.na.2022.112822},
url = {https://www.sciencedirect.com/science/article/pii/S0362546X22000293},
author = {Giuliano Lazzaroni and Riccardo Molinarolo and Francesco Solombrino}
}

@inproceedings{zolesio1990approximation,
author={Zolesio, Jean Paul},
editor={Zolesio, Jean Paul},
title={Galerkine approximation for wave equation in moving domain},
booktitle={Stabilization of Flexible Structures},
year={1990},
publisher={Springer Berlin Heidelberg},
address={Berlin, Heidelberg},
pages={191--225},
isbn={978-3-540-46731-1}
}

@article{dal2017wave,
  title={The wave equation on domains with cracks growing on a prescribed path: existence, uniqueness, and continuous dependence on the data},
  author={Dal Maso, Gianni and Lucardesi, Ilaria},
  journal={Applied Mathematics Research eXpress},
  volume={2017},
  number={1},
  pages={184--241},
  year={2017},
  publisher={Oxford University Press}
}

@article{dal2016existence,
  title={Existence and uniqueness of dynamic evolutions for a peeling test in dimension one},
  author={Dal Maso, Gianni and Lazzaroni, Giuliano and Nardini, Lorenzo},
  journal={Journal of Differential Equations},
  volume={261},
  number={9},
  pages={4897--4923},
  year={2016},
  publisher={Elsevier}
}

@article{bernardi1998some,
  title={On some abstract variable domain hyperbolic differential equations},
  author={Bernardi, Marco Luigi and Bonfanti, Giovanna and Luterotti, Fabio},
  journal={Annali di Matematica Pura ed Applicata},
  volume={174},
  pages={209--239},
  year={1998},
  publisher={Springer}
}

@article{cooper1973nonlinear,
  title={A nonlinear wave equation in a time dependent domain},
  author={Cooper, Jeffery and Bardos, Claude},
  journal={Journal of Mathematical Analysis and Applications},
  volume={42},
  number={1},
  pages={29--60},
  year={1973},
  publisher={Academic Press}
}

@article{del2020interface,
  title={Interface dynamics in semilinear wave equations},
  author={Del Pino, Manuel and Jerrard, Robert Leon and Musso, Monica},
  journal={Communications in Mathematical Physics},
  volume={373},
  number={3},
  pages={971--1009},
  year={2020},
  publisher={Springer}
}

@article{dal2002model,
  title={A model for the Quasi-Static growth of brittle fractures: existence and approximation results},
  author={Dal Maso, Gianni and Toader, Rodica},
  journal={Archive for Rational Mechanics and Analysis},
  volume={162},
  pages={101--135},
  year={2002},
  publisher={Springer}
}

@article{dal2011existence,
  title={Existence for wave equations on domains with arbitrary growing cracks},
  author={Dal Maso, Gianni and Larsen, Christopher J},
  journal={Rendiconti Lincei},
  volume={22},
  number={3},
  pages={387--408},
  year={2011}
}

@inproceedings{da1990existence,
  title={Existence and optimal control for wave equation in moving domain},
  author={Da Prato, Giuseppe and Zol{\'e}sio, Jean-Paul},
  booktitle={Stabilization of Flexible Structures: Third Working Conference Montpellier, France, January 1989},
  pages={167--190},
  year={1990},
  organization={Springer}
}

@article{ma2017dynamics,
  title={Dynamics of wave equations with moving boundary},
  author={Ma, To Fu and Mar{\'\i}n-Rubio, Pedro and Chu{\~n}o, Christian Manuel Surco},
  journal={Journal of Differential Equations},
  volume={262},
  number={5},
  pages={3317--3342},
  year={2017},
  publisher={Elsevier}
}

@article{sikorav1990linear,
  title={A linear wave equation in a time-dependent domain},
  author={Sikorav, Jean Claude},
  journal={Journal of mathematical analysis and applications},
  volume={153},
  number={2},
  pages={533--548},
  year={1990},
  publisher={Elsevier}
}

@article{alphonse2015abstract,
  title={An abstract framework for parabolic PDEs on evolving spaces},
  author={Alphonse, Amal and Elliott, Charles M and Stinner, Bj{\"o}rn},
  journal={Portugaliae Mathematica},
  volume={72},
  number={1},
  pages={1--46},
  year={2015}
}

@article{bernardi2001variational,
  title={Variational equations of Schroedinger-type in non-cylindrical domains},
  author={Bernardi, Marco Luigi and Pozzi, Gianni Arrigo and Savar{\'e}, Giuseppe},
  journal={Journal of Differential Equations},
  volume={171},
  number={1},
  pages={63--87},
  year={2001},
  publisher={Elsevier}
}

@article{gianazza1996abstract,
  title={Abstract evolution equations on variable domains: an approach by minimizing movements},
  author={Gianazza, Ugo and Savar{\'e}, Giuseppe},
  journal={Annali della Scuola Normale Superiore di Pisa-Classe di Scienze},
  volume={23},
  number={1},
  pages={149--178},
  year={1996}
}

@article{horvath2020exactly,
  title={An exactly mass conserving space-time embedded-hybridized discontinuous Galerkin method for the Navier--Stokes equations on moving domains},
  author={Horv{\'a}th, Tam{\'a}s and Rhebergen, Sander},
  journal={Journal of Computational Physics},
  volume={417},
  pages={109577},
  year={2020},
  publisher={Elsevier}
}

@book{hormander1997lectures,
  title={Lectures on nonlinear hyperbolic differential equations},
  author={H{\"o}rmander, Lars},
  volume={26},
  year={1997},
  publisher={Springer Science \& Business Media}
}

@book{lax2006hyperbolic,
  title={Hyperbolic partial differential equations},
  author={Lax, Peter David},
  volume={14},
  year={2006},
  publisher={American Mathematical Soc.}
}

@book{lionsmagenes,
  title={Non-Homogeneous Boundary Value Problems and Applications},
  author={Lions, Jacques-Louis and Magenes, Enrico},
  volume={1},
  year={1972},
  publisher={Berlin, Grundlehren der Mathematischen Wissenschaften, Springer-Verlag,}
}

@article{knobloch2015problems,
  title={Problems on time-varying domains: Formulation, dynamics, and challenges},
  author={Knobloch, Edgar and Krechetnikov, Rouslan},
  journal={Acta Applicandae Mathematicae},
  volume={137},
  number={1},
  pages={123--157},
  year={2015},
  publisher={Springer}
}

@article{savare1997parabolic,
  title={Parabolic problems with mixed variable lateral conditions: an abstract approach},
  author={Savar{\'e}, Giuseppe},
  journal={Journal de math{\'e}matiques pures et appliqu{\'e}es},
  volume={76},
  number={4},
  pages={321--351},
  year={1997},
  publisher={Elsevier}
}

@article{bonaccorsi2001variational,
  title={A variational approach to evolution problems with variable domains},
  author={Bonaccorsi, Stefano and Guatteri, Giuseppina},
  journal={Journal of Differential Equations},
  volume={175},
  number={1},
  pages={51--70},
  year={2001},
  publisher={Elsevier}
}

@article{paronetto2013existence,
  title={An existence result for evolution equations in non-cylindrical domains},
  author={Paronetto, Fabio},
  journal={Nonlinear Differential Equations and Applications NoDEA},
  volume={20},
  number={6},
  pages={1723--1740},
  year={2013},
  publisher={Springer}
}

@article{dalla2023multi,
title = {Multi-parameter perturbations for the space-periodic heat equation},
journal = {Communications on Pure and Applied Analysis},
volume = {23},
number = {2},
pages = {144-164},
year = {2024},
issn = {1534-0392},
doi = {10.3934/cpaa.2024004},
url = {https://www.aimsciences.org/article/id/65b2230f9ce8685d70f3270c},
author = {Dalla Riva, Matteo and Luzzini, Paolo and Molinarolo, Riccardo and Musolino, Paolo}
}

@article {dalla2025shape,
    AUTHOR = {Dalla Riva, Matteo and Luzzini, Paolo and Molinarolo, Riccardo
              and Musolino, Paolo},
     TITLE = {Shape perturbation of a nonlinear mixed problem for the heat
              equation},
   JOURNAL = {J. Evol. Equ.},
  FJOURNAL = {Journal of Evolution Equations},
    VOLUME = {25},
      YEAR = {2025},
    NUMBER = {1},
     PAGES = {Paper No. 18, 25},
      ISSN = {1424-3199,1424-3202},
   MRCLASS = {35K20 (31B10 45A05 47H30)},
  MRNUMBER = {4844095},
MRREVIEWER = {Seonghak\ Kim},
       DOI = {10.1007/s00028-024-01047-5},
       URL = {https://doi.org/10.1007/s00028-024-01047-5},
}

@article{Grillakis,
  title={Regularity and Asymptotic Behavior of the Wave Equation with a Critical Nonlinearity},
  author={Grillakis, Manoussos},
  journal={Annals of Mathematics},
  volume={132},
  number={3},
  pages={485-509},
  year={1990},
}

@article{ColomboHaffter,
  title={Global Regularity For The Nonlinear Wave Equation With Slightly Supercritical Power},
  author={Colombo, Maria and Haffter, Silja},
  journal={Analysis \& PDEs},
  volume={16},
  number={3},
  pages={613-642},
  year={2023},
}

@article{Lions64,
  title={Une remarque sur les problèmes d’évolution non linéaires dans des domaines non cylindriques},
  author={Lions, Jacques-Louis},
  journal={Revue Roumaine de Mathématiques Pures et Appliquées},
  volume={9},
  pages={11-18},
  year={1964},
}

@article{FerreiraRaposoSantos2006,
  title={Global existence for a quasilinear hyperbolic equation in a
noncylindrical domain},
  author={Ferreira, Jorge and Raposo, Carlos Alberto and  Santos, Mauro Lima},
  journal={International Journal of Pure and Applied Mathematics},
  volume={29},
    number={4},
  pages={457-467},
  year={2006},
}

@article{Lopez2015Remarks,
  author    = {López, Ivo Fernandez and  Antunes, Gladson Octaviano and da Silva,  Maria Darci Godinho and  Medeiros, Luiz Adauto da Justa},
  title     = {Remarks on a Nonlinear Wave Equation in a Noncylindrical Domain},
  journal   = {Tendências em Matemática Aplicada e Computacional (TEMA)},
  year      = {2015},
  volume    = {16},
  number    = {3},
  pages     = {195--207},
}

@article{MR1727557,
    AUTHOR = {Silling, Stewart Andrew},
     TITLE = {Reformulation of elasticity theory for discontinuities and
              long-range forces},
   JOURNAL = {J. Mech. Phys. Solids},
  FJOURNAL = {Journal of the Mechanics and Physics of Solids},
    VOLUME = {48},
      YEAR = {2000},
    NUMBER = {1},
     PAGES = {175--209},
      ISSN = {0022-5096,1873-4782},
   MRCLASS = {74A99 (74R10)},
  MRNUMBER = {1727557},
MRREVIEWER = {Paul\ Andrew\ Martin},
       DOI = {10.1016/S0022-5096(99)00029-0},
       URL = {https://doi.org/10.1016/S0022-5096(99)00029-0},
}

@article {MR3403446,
    AUTHOR = {Emmrich, Etienne and Puhst, Dimitri},
     TITLE = {Survey of existence results in nonlinear peridynamics in
              comparison with local elastodynamics},
   JOURNAL = {Comput. Methods Appl. Math.},
  FJOURNAL = {Computational Methods in Applied Mathematics},
    VOLUME = {15},
      YEAR = {2015},
    NUMBER = {4},
     PAGES = {483--496},
      ISSN = {1609-4840,1609-9389},
   MRCLASS = {74B20 (35D30 35L72 35Q74 74H20)},
  MRNUMBER = {3403446},
MRREVIEWER = {Isabelle\ Gruais},
       DOI = {10.1515/cmam-2015-0020},
       URL = {https://doi.org/10.1515/cmam-2015-0020},
}

@article {MR3297136,
    AUTHOR = {Emmrich, Etienne and Puhst, Dimitri},
     TITLE = {Measure-valued and weak solutions to the nonlinear peridynamic
              model in nonlocal elastodynamics},
   JOURNAL = {Nonlinearity},
  FJOURNAL = {Nonlinearity},
    VOLUME = {28},
      YEAR = {2015},
    NUMBER = {1},
     PAGES = {285--307},
      ISSN = {0951-7715,1361-6544},
   MRCLASS = {35Q74 (35D30 74B20 74H20)},
  MRNUMBER = {3297136},
MRREVIEWER = {Ramon\ Quintanilla},
       DOI = {10.1088/0951-7715/28/1/285},
       URL = {https://doi.org/10.1088/0951-7715/28/1/285},
}

@article {MR4648534,
    AUTHOR = {Coclite, Alessandro and Coclite, Giuseppe Maria and Fanizza,
              Giuseppe and Maddalena, Francesco},
     TITLE = {Dispersive effects in two- and three-dimensional peridynamics},
   JOURNAL = {Acta Appl. Math.},
  FJOURNAL = {Acta Applicandae Mathematicae},
    VOLUME = {187},
      YEAR = {2023},
     PAGES = {13--21},
      ISSN = {0167-8019,1572-9036},
   MRCLASS = {74A70 (35Q70 70G70 74B10)},
  MRNUMBER = {4648534},
       DOI = {10.1007/s10440-023-00606-1},
       URL = {https://doi.org/10.1007/s10440-023-00606-1},
}

@article {MR4500878,
    AUTHOR = {Coclite, Giuseppe Maria and Dipierro, Serena and Fanizza,
              Giuseppe and Maddalena, Francesco and Valdinoci, Enrico},
     TITLE = {Dispersive effects in a scalar nonlocal wave equation inspired
              by peridynamics},
   JOURNAL = {Nonlinearity},
  FJOURNAL = {Nonlinearity},
    VOLUME = {35},
      YEAR = {2022},
    NUMBER = {11},
     PAGES = {5664--5713},
      ISSN = {0951-7715,1361-6544},
   MRCLASS = {74A70 (35Q70 70G70)},
  MRNUMBER = {4500878},
MRREVIEWER = {Giuseppe\ Saccomandi},
       DOI = {10.1088/1361-6544/ac8fd9},
       URL = {https://doi.org/10.1088/1361-6544/ac8fd9},
}

@article {MR3884619,
    AUTHOR = {Coclite, Giuseppe Maria and Dipierro, Serena and Maddalena,
              Francesco and Valdinoci, Enrico},
     TITLE = {Wellposedness of a nonlinear peridynamic model},
   JOURNAL = {Nonlinearity},
  FJOURNAL = {Nonlinearity},
    VOLUME = {32},
      YEAR = {2019},
    NUMBER = {1},
     PAGES = {1--21},
      ISSN = {0951-7715,1361-6544},
   MRCLASS = {35R09 (35B25 35B30 35L52 74K30 74K35 74R10)},
  MRNUMBER = {3884619},
       DOI = {10.1088/1361-6544/aae71b},
       URL = {https://doi.org/10.1088/1361-6544/aae71b},
}
\end{document}